\documentclass[12pt, twoside, leqno]{article}

\usepackage{amssymb}
\usepackage{amsmath}
\usepackage{amsthm}
\usepackage{color}
\usepackage{mathrsfs}
\usepackage{esint}

\allowdisplaybreaks

\newtheorem{theorem}{Theorem}[section]
\newtheorem{lemma}{Lemma}[section]
\newtheorem{corollary}{Corollary}[section]

\theoremstyle{definition}
\newtheorem{remark}{Remark}[section]
\newtheorem{definition}{Definition}[section]
\numberwithin{equation}{section}

\def\beq{\begin{equation}}
\def\eeq{\end{equation}}
\def\ba{\begin{array}}
\def\ea{\end{array}}

\def\R{\mathbb R}

\theoremstyle{definition}

\theoremstyle{remark}

\begin{document}
\arraycolsep=1pt
\title{Anisotropic Moser-Trudinger inequality involving $L^{n}$ norm in the entire space $\mathbb{R}^{n}$
\footnotetext{\hspace{-0.35cm}
2010 {\it Mathematics Subject Classification}. 35A15; 46E35.
\endgraf {\it Key words and phrases}: Moser-Trudinger inequality; anisotropic Sobolev norm;
Blow up analysis; extremal function; unbounded domain.
\endgraf {\it E-mail}: xierl@ustc.edu.cn}}
\author{Rulong Xie\\ \small\small {School of Mathematical Sciences,
University of Science and Technology of China, Hefei 236000, China}
\small { }\vspace{0.3cm}}
\date{} \maketitle

\begin{abstract}
Let $F: \mathbb{R}^{n}\rightarrow [0,+\infty) $ be a convex function of class $C^{2}( \mathbb{R}^{n}\backslash\{0\})$ which is even and positively homogeneous of degree 1, and its polar $F^{0}$ represents a Finsler metric on $\mathbb{R}^{n}$. The anisotropic Sobolev norm in $W^{1,n}\left(\mathbb{R}^{n}\right)$ is defined by
\begin{equation*}
||u||_{F}=\left(\int_{\mathbb{R}^{n}}F^{n}(\nabla u)+|u|^{n}\right)^{\frac{1}{n}}.
\end{equation*} In this paper, the following sharp
 anisotropic Moser-Trudinger inequality involving $L^{n}$ norm
\[
\underset{u\in W^{1,n}(
\mathbb{R}^{n}),\left\Vert u\right\Vert _{F}\leq 1}{\sup}\int_{
\mathbb{R}
^{n}}\Phi\left( \lambda_{n}\left\vert u\right\vert ^{\frac{n}{n-1}}\left(
1+\alpha\left\Vert u\right\Vert _{n}^{n}\right)  ^{\frac{1}{n-1}}\right)
dx<+\infty
\]
in the entire space $\mathbb{R}^n$ for any $0\leq\alpha<1$ is established, where $\Phi\left(  t\right)  =e^{t}-\underset{j=0}{\overset{n-2}{\sum}}%
\frac{t^{j}}{j!}$, $\lambda_{n}=n^{\frac{n}{n-1}}\kappa_{n}^{\frac{1}{n-1}}$ and $\kappa_{n}$ is the volume of the unit Wulff ball
in $\mathbb{R}^n$. It is also shown that  the above supremum is infinity for all
$\alpha\geq1$. Moreover, we prove the supremum is attained, namely, there exists a maximizer for the above supremum when $\alpha>0$ is
sufficiently small. The proof of main results in this paper is based on the method of blow-up analysis.

\end{abstract}

\section{Introduction}

 Let $\Omega\subset \R^n $ be a smooth bounded domain. It is well known that $W_{0}^{1,n}(\Omega)$ is embedded into $L^{p}(\Omega)$ for any $p>1$.  Namely, using the Dirichlet norm $\|u\|_{W^{1,n}_0(\Omega)}=(\int_{\Omega}|\nabla u|^ndx)^{\frac 1n} $  on $W_{0}^{1,n}(\Omega)$, we have
  $$\sup_{u\in W_{0}^{1,n}(\Omega),||\nabla u||_{L^{n}(\Omega)}\leq1}\int_{\Omega}|u|^pdx<+\infty.$$
  But $W_{0}^{1,n}(\Omega)$ is not embedded into $ L^{\infty}(\Omega)$. Hence, many mathematical researchers would like to look for a function $g(s): \R \rightarrow \R^+ $ with maximal growth such that
  $$\sup_{u\in W_{0}^{1,n}(\Omega),||\nabla u||_{L^{n}(\Omega)}\leq1}\int_{\Omega}g(u)dx<+\infty.$$  The Moser-Trudinger inequality states that the maximal growth function is of exponential type, which was shown by Pohozhaev \cite{S2}, Trudinger \cite{Trudinger} and Moser \cite{moser}. This inequality says that
\begin{equation}\label{moser-tru}
\sup_{u\in W_{0}^{1,n}(\Omega),||\nabla u||_{L^{n}(\Omega)}\leq1}\int_{\Omega}e^{\alpha |u|^{\frac{n}{n-1}}}dx<+\infty
\end{equation}
for any $\alpha\leq n\omega_{n-1}^{\frac{1}{n-1}}:=\alpha_{n}$, where $\omega_{n-1}$ is the surface area of the unit ball in $\R^n$. Also the inequality is optimal, that is, for any $\alpha>\alpha_{n}$ there exists a sequence of $\{u_{\epsilon}\}$ in $W_{0}^{1,n}(\Omega)$ with  $||\nabla u_\epsilon||_{L^{n}(\Omega)}\leq 1 $ such that
\begin{equation*}
\int_{\Omega}e^{\alpha |u_{\epsilon}|^{\frac{n}{n-1}}}dx\rightarrow +\infty \qquad as~~\epsilon\rightarrow 0.
\end{equation*}

Whether extremal functions  exist or not is another interesting question about Moser-Trudinger inequality.  Carleson and Chang \cite{CC} firstly proved that the supremum is attained when $\Omega$ is a unit ball in $\R^{n}$. Then Struwe \cite{S} got the existence of extremals for $\Omega$ close to a ball. Struwe's technique was then used and extended by Flucher \cite{F} to $\Omega$ which is the more general bounded smooth domain in $\R^2$. Later, Lin \cite{lin} generalized the existence result to a bounded smooth domain in $\R^n$.

Numerous generalizations, extensions and applications of the Moser-Trudinger
inequality have been obtained due to important applications in partial differential equations and geometric analysis (see \cite{Adams}-\cite{Adimurthi}, \cite{J. M. do}-\cite{J. M. do2}, \cite{Figueiredo}-\cite{Malchiodi}, \cite{MM}, \cite{de}, \cite{yang1}-\cite{YZ1} and references therein). We recall in particular the famous concentration-compactness result
obtained by Lions \cite{lions}, which says that if $\left\{
u_{k}\right\}  $ is a sequence of functions in $W_{0}^{1,n}\left(
\Omega\right)  $\ with $\left\Vert \nabla u_{k}\right\Vert _{L^n(\Omega)}=1$
such that $u_{k}\rightharpoonup u$ weakly in $W^{1,n}\left(
\Omega\right)  $, then for any $0<p<\left( 1-\left\Vert \nabla
u\right\Vert
_{L^n(\Omega)}^{n}\right)  ^{-1/\left(  n-1\right)}$, it follows

\[
\underset{k\rightarrow \infty}{\sup}\int_{ \Omega}e^{\alpha_{n}p\left\vert u_{k}\right\vert
^{\frac{n}{n-1}}}dx<+\infty.
\]
 Based on the result of Lions
and the blowing up analysis method, Adimurthi and Druet
\cite{Adimurthi} obtained an improved Moser-Trudinger type
inequality in $\mathbb{R}
^{2}$ on bounded domains $\Omega$, which can be described as follows \[
\underset{\left\Vert \nabla u\right\Vert _{2}\leq1,u\in W_{0}
^{1,2}\left( \Omega\right)}{\sup}\int_{
\mathbb{R}
^{2}}e^{  4\pi\left\vert u\right\vert ^{2}\left(
1+\alpha\left\Vert u\right\Vert _{2}^{2}\right)}
dx<+\infty\ \text{if and only if }\alpha<\underset{ u\in W_{0}^{1,2}\left(
\Omega\right) ,u\neq0}{\inf}\frac{\left\Vert \nabla u\right\Vert
_{2}^{2}}{\left\Vert u\right\Vert _{2}^{2}}.
\]  Later, this result was extended to high dimension and $L^p$ norm in two dimension or high dimension in Yang \cite{yang1}, Lu-Yang \cite{lu-yang,lu-yang 1} and Zhu \cite{zhu}.

Related Moser-Trudinger inequalities for unbounded domains have been first considered by
Cao \cite{cao} in dimension two and for any dimension by
do \'{O} \cite{J. M. do2} and Adachi-Tanaka \cite{Adachi-Tanaka}.
In \cite{ruf}, Ruf showed that in the case of dimension two, one obtains that

\begin{equation}
\underset{\int_{\mathbb{R}
^{2}}\left(  \left\vert u\right\vert ^{2}+\left\vert \nabla
u\right\vert ^{2}\right)  dx\leq1,u\in W^{1,2}\left(
\mathbb{R}
^{2}\right)  }{\sup}\int_{
\mathbb{R}
^{2}}\phi\left(  \alpha\left\vert u\right\vert ^{2}\right)  dx<
+\infty\ \text{ if and only if }\alpha\leq4\pi,
\end{equation}
where $\phi\left(  t\right)  =e^{t}-1$. Li and Ruf
\cite{liruf} extended Ruf's result to arbitrary dimension.
Later, Souza and do \'{O} \cite{do1} obtained an Adimurthi-Druet type result in $\mathbb{R}^2$ for some weighted Sobolev space. Recently, Lu and Zhu \cite{LuZhu} proved a sharp Moser-Trudinger inequality involving
$L^{n}$ norm in $\mathbb{R}^{n}$.

The one interesting extension of (\ref{moser-tru}) is to establish anisotropic Moser-Trudinger inequality which involves $n$-anisotropic Laplacian (or $n$-Finsler Laplacian) $Q_{n}$ as follows:

\begin{equation*}
Q_{n} u:=\sum_{i=1}^{n}\frac{\partial}{\partial x_{i}}(F^{n-1}(\nabla u)F_{\xi_i}(\nabla u)).
\end{equation*}
Here the function $F(x)$ is positive, convex and homogeneous of degree $1$, and its polar $F^0(x)$ represents a Finsler metric on $\R^n$. The properties of the operator $Q_{n}$ was researched by Gong and the author of this paper in \cite{XG}.

In 2012,  Wang and  Xia \cite{WX} proved the following anisotropic Moser-Trudinger inequality
\begin{equation}\label{1-02}
\int_{\Omega}e^{\lambda |u|^{\frac{n}{n-1}}}dx\leq C(n)|\Omega|
\end{equation}
for all $u\in W_{0}^{1,n}(\Omega)$ and $\int_{\Omega}F^{n}(\nabla u)dx\leq 1$. Here $\lambda\leq\lambda_{n}:=n^{\frac{n}{n-1}}\kappa_n^{\frac{1}{n-1}}$, where $\kappa_{n}$ is the volume of the unit Wulff ball
in $\mathbb{R}^n$, i.e. $\kappa_{n}=|\{x\in \mathbb{R}^{n}: F^{0}(x)\leq 1\}|$. $\lambda_{n}$ is optimal in the sense that if $\lambda >\lambda_{n}$ one can find a sequence $\{u_{k}\}$ such that $\int_{\Omega}e^{\lambda |u_k|^{\frac{n}{n-1}}}dx$ diverges. Later, Zhou and Zhou \cite{ZZ,ZZ1} have shown that  the supremum is attained when $\Omega$  is bounded domain in  $\R^{n}$. Recently, Zhou \cite{Z1} obtained the anisotropic Moser-Trudinger inequality involving $L^{n}$ norm in a smooth bounded domain $\Omega\in\R^{n}$ and Liu \cite{Liu} extended the corresponding result to $L^{p}$ norm. On the unbounded domain in $\R^{n}$, Zhou and Zhou \cite{ZZ2} established the anisotropic Moser-Trudinger inequality.

In this paper, we will research the anisotropic Moser-Trudinger type inequality involving $L^{n}$ norm and its extremal functions in the entire space $\R^{n}$. The isotropic Dirichlet norm  $\|u\|_{W^{1,n}_0(\Omega)}=(\int_{\Omega}|\nabla u|^ndx)^{\frac 1n} $ will be replaced  by the anisotropic Dirichlet norm  $(\int_{\Omega}F^n(\nabla u)dx)^{\frac 1n} $  on $W_{0}^{1,n}(\Omega)$. Also, the isotropic Sobolev norm will be replaced by the anisotropic Sobolev norm

\begin{equation*}
||u||_{F}=\left(\int_{\R^{n}}F^{n}(\nabla u)+|u|^{n}\right)^{\frac{1}{n}}.
\end{equation*}
Now we stated the main results in this paper as follows.
\begin{theorem}
\label{moser-trudinger} For any $0\leq\alpha<1$, we have

\begin{equation}
\underset{u\in W^{1,n}(
\mathbb{R}^{n}),\left\Vert u\right\Vert _{F}\leq 1}{\sup}\int_{
\mathbb{R}
^{n}}\Phi\left(  \lambda_{n}\left\vert u\right\vert
^{\frac{n}{n-1}}\left( 1+\alpha\left\Vert u\right\Vert
_{n}^{n}\right)  ^{\frac{1}{n-1}}\right) dx<+\infty,
\label{moser-Trudi}\end{equation}
where $\Phi\left(  t\right)  =e^{t}-\underset{j=0}{\overset{n-2}{\sum}}%
\frac{t^{j}}{j!}$. Moreover,   for any $\alpha\geq1,$ the
supremum is infinite.
\end{theorem}

Set
\[S=\underset{u\in W^{1,n}(
\mathbb{R}^{n}),\left\Vert u\right\Vert _{F}\leq 1}{\sup}\int_{
\mathbb{R}
^{n}}\Phi\left(  \lambda_{n}\left\vert u\right\vert ^{\frac{n}{n-1}}\left(
1+\alpha\left\Vert u\right\Vert _{n}^{n}\right)  ^{\frac{1}{n-1}}\right)  dx.
\]

\begin{theorem}
\label{attain}There exists $u_{\alpha}\in W^{1,n}\left(
\mathbb{R}
^{n}\right)  $ with $\left\Vert u_{\alpha}\right\Vert _{F}=1$ such that
\[
S=\int_{\mathbb{R}
^{n}}\Phi\left(  \lambda_{n}\left\vert u_{\alpha}\right\vert ^{\frac{n}{n-1}%
}\left(  1+\alpha\left\Vert u_{\alpha}\right\Vert _{n}^{n}\right)  ^{\frac
{1}{n-1}}\right)  dx
\]
for sufficiently small $\alpha$.
\end{theorem}

\medskip

This paper is organized as follows. In Section 2 we recall some notations and preliminaries which will be use later.
Section 3 is devoted to proving the existence of radially symmetric maximizing sequence for the critical
functional. In Section 4 we give the proof of Theorem 1.1. We prove the sharpness
of the inequality in Theorem \ref{moser-trudinger}, i.e. the second part of Therorem \ref{moser-trudinger} by constructing a
appropriate test function sequence in Subsection 4.1. In Subsection 4.2, we prove the first part of Theorem \ref{moser-trudinger} by considering the two cases. In Subsection 4.2.1, we prove the first part of Theorem \ref{moser-trudinger} in the case of $sup_{k}c_{k}<+\infty$. The proof in the case of $sup_{k}c_{k}=+\infty$ is arranged in Subsection 4.2.2, we apply the blowing up analysis to
analyze the asymptotic behavior of the maximizing sequence near and
far away from the origin, and give the proof of the first part of
Theorem \ref{moser-trudinger} in this case. In Section 5, we also prove Theorem \ref{attain} by considering the two cases, which are $sup_{k}c_{k}<+\infty$ and $sup_{k}c_{k}=+\infty$. In Subsection 5.1, based on the concentration-compactness lemma, we give the proof of Theorem \ref{attain}  in the case of $sup_{k}c_{k}<+\infty$. In Subsection 5.2, we prove the result in case of $sup_{k}c_{k}=+\infty$ by contradiction.
For this, we first establish the upper bound for critical functional when $\sup_{k}c_{k}=+\infty$, and then construct an explicit test function, which provides a lower bound for the
supremum of our Moser-Trudinger inequality. Because this lower bound equals to the upper bound, one can obtain the contradiction and prove Theorem \ref{attain} in this case.

Throughout this paper, the letter $C$ denotes a constant independent of the main functions which may
be different from line to line.

\section{Notations and preliminaries }
In this section, let us recall some important concepts and preliminaries which will be use later in this paper.

Throughout this paper, let  $F :\R^n\mapsto \R$  be a nonnegative convex function of class $C^{2}(\R^n\backslash\{0\})$  which is even and positively homogenous of degree  $1$,  i.e.
$$F(t\xi)=|t|F(\xi)\qquad \text{for any}\qquad t\in \R,~~~~\xi\in \R^n.$$
A typical example is  $F(\xi)=(\sum_{i}|\xi_i|^{q})^{\frac{1}{q}}$  for  $q\in [1,\infty)$.  We further assume that
$$F(\xi)>0\qquad \text{for any}\qquad\xi\neq 0.$$

With the help of homogeneity of  $F$,  there exist two constants  $0<a\leq b<\infty$  such that
\begin{equation*}
a|\xi|\leq F(\xi)\leq b|\xi|.
\end{equation*}
 Usually, we shall assume that the $Hess(F^{2})$ is positive definite in  $\R^n\backslash\{0\}$. Then by Gong and the author of this paper in \cite{XG}, $Hess(F^{n})$ is also positive definite in  $\R^n\backslash\{0\}$.
Considering the minimization problem
$$
\min_{u\in W^{1,n}_0(\Omega)}\int_{\Omega}F^n(\nabla u)dx,
$$
its Euler-Lagrange equation contains an  operator of the form
$$
Q_{n}u:=\sum_{i=1}^{n}\frac{\partial}{\partial x_i}(F^{n-1}(\nabla u)F_{\xi_i}(\nabla u)),
$$ which is called as n-anisotropic Laplacian or n-Finsler Laplacian.

 It is well known that in the isotropic case, i.e. $F(\xi) = |\xi|$, when $n=2$, $Q_n$ is the ordinary Laplacian operator; when $n>2$, $Q_n$ is the $n$-Laplacian operator. In the anisotropic case, when $n=2$, $Q_{n}$ is  anisotropic Laplacian operator. The operator $Q_{n}$ was studied by many researchers, see \cite{FK,WX,AVP,BFK,XG} and their references therein.

Consider the map
$$\phi:S^{n-1}\rightarrow \R^n, ~~~~~ \phi(\xi)=F_{\xi}(\xi).$$
Its image  $\phi(S^{n-1})$  is a smooth, convex hypersurface in  $\R^n$, which is called Wulff shape of  $F$.
Let  $F^{o}$  be the support function of  $K:=\{x\in \R^n:F(x)\leq 1\}$,  which is defined by
$$F^{o}(x):=\sup_{\xi\in K}\langle x,\xi\rangle.$$
It is easy to prove that  $F^{o}:\R^n\mapsto [0,+\infty)$  is also a convex, homogeneous function of class of $C^{2}(\R^n\backslash\{0\})$.  Actually  $F, F^{0}$ are polar to each other in the sense that
$$F^{o}(x)=\sup_{\xi\neq 0}\frac{\langle x,\xi\rangle}{F(\xi)},\qquad F(x)=\sup_{\xi\neq 0}\frac{\langle x,\xi\rangle}{F^{o}(\xi)}.$$
One can see easily that  $\phi(S^{n-1})=\{x\in \R^n~|F^{o}(x)=1\}$. Let $\mathcal{W}_{F}:=\{x\in \mathbb{R}^{n}: F^{0}(x)\leq 1\}$ and $\kappa_{n}=|\mathcal{W}_{F}|$, which is the Lebesgue measure of $\mathcal{W}_{F}$. Also, denote $\mathcal{W}_{r}(x_{0})$ by the Wulff ball of center at $x_{0}$ with radius $r$, i.e. $\mathcal{W}_{r}(x_{0})=\{x\in \mathbb{R}^{n}: F^{0}(x-x_{0})\leq r\}$.

 Next, we summarize the properties on $F$ and $F^{0}$, which can be proved easily by the assumption on $F$, also see \cite{WX1,FK,BP}.

\begin{lemma}\label{2-01}
We have
\begin{enumerate}
\item[(i)]$|F(x)-F(y)|\leq F(x+y)\leq F(x)+F(y)$;
\item[(ii)]$ \frac{1}{C}\leq|\nabla F(x)|\leq C $, and $ \frac{1}{C}\leq|\nabla F^{o}(x)|\leq C $  for some $C>0$ and any $x\neq 0$;
\item[(iii)]$\langle \xi,\nabla F(\xi)\rangle=F(\xi),\langle x,\nabla F^{o}(x)\rangle=F^{o}(x)$ for any $x\neq 0,\ \xi\neq 0$;
\item[(iv)]$F(\nabla F^{o}(x))=1$, $F^{o}(\nabla F(\xi))=1$ for any $x\neq 0,\ \xi\neq 0$;
\item[(v)]$F^{o}(x) F_{\xi}(\nabla F^{o}(x))=x $ for any $ x\neq 0$;
\item[(vi)]$F_\xi(t\xi)= \text{sgn}(t)F_{\xi}(\xi)$ for any $\xi\neq 0$ and $t\neq 0$.
\end{enumerate}
\end{lemma}
Next we give the co-area formula and isoperimetric inequality in the anisotropic situation.

For a bounded domain $\Omega\subset \mathbb{R}^{n}$ and a function of bounded variation $u\in BV(\Omega)$, denote the anisotropic bounded variation of $u$ with respect to $F$ by
$$\int_{\Omega}|\nabla u|_{F}=\sup\left\{\int_{\Omega}u\text{div}\sigma dx: \sigma \in C_{0}^{1}(\Omega),F^{0}(\sigma)\leq 1\right\},$$
and anisotropic perimeter of $E$  with respect to $F$ by
$$P_{F}(E):=\int_{\Omega}|\nabla \chi_{E}|_{F},$$
 where $E$ is a subset of $\Omega$ and $\chi_{E}$ is the characeristic function of $E$.
The co-area formula and isoperimetric inequality can be expressed by
\begin{equation}\label{2-02}
\int_{\Omega}|\nabla u|_{F}=\int_{0}^{\infty}P_{F}(|u|>t)dt,
\end{equation}
and
\begin{equation}\label{2-03}
P_{F}(E)\geq Nk^{\frac{1}{N}}|E|^{1-\frac{1}{N}}
\end{equation}
respectively.
Moreover, the equality in (\ref{2-03}) holds if and only if  $E$  is a Wulff ball.

In the sequel, we will use the convex symmetrization with respect to $F$. The convex symmetrization generalizes the Schwarz symmetrization (see \cite{T3}). It was defined in \cite{AVP} and will be an essential tool for establishing  the Lions type concentration-compactness theorem under the anisotropic Dirichlet norm.
Let us consider a measurable function  $u$  on  $\Omega\subset \R^n$.  The one dimensional decreasing rearrangement of $u$  is defined as
$$u^{*}=\sup\{s\geq 0: |\{x\in\Omega:|u(x)|>s\}|>t\},\qquad \text{for} \quad  t\in \R.$$
The convex symmetrization of  $u$  with respect to  $F$  is
$$u^{\star}(x)=u^{*}(\kappa_{n} F^{o}(x)^{n}),\qquad\text{ for } x\in \Omega^{*}.$$
Here  $\kappa_{n} F^{o}(x)^{n}$  is just Lebesgue measure of a homothetic Wulff ball with radius $F^{0}(x)$ and $\Omega^{*}$  is the homothetic Wulff ball centered at the origin having the same measure as $\Omega$. Throughout this paper, we assume that $\Omega$ is bounded smooth domain in $\mathbb{R}^{n}$ with $n\geq 2$.

Now let us recall some important results which can be found in \cite{ZZ1,ZZ2}.  Lemma \ref{jmdo} is also called the concentration-compactness lemma.
\begin{lemma}\label{lem2.2}
Assume that $u\in W_{0}^{1,n}(\Omega)$ is a solution of the equation
\begin{equation}\label{2-05}
-Q_{n}(u)=f.
\end{equation}
If $f\in L^{q}(\Omega)$ for some $q>1$, then $||u||_{L^{\infty}(\Omega)}\leq C||f||_{L^{q}(\Omega)}^{\frac{1}{n-1}}$.
\end{lemma}

\begin{lemma}
\label{jmdo} Let $\left\{  u_{k}\right\}  $ be a sequence in $W^{1,n}\left(
\mathbb{R}
^{n}\right)  $ such that $\left\Vert u_{k}\right\Vert _{F}=1$ and
$u_{k}\rightharpoonup u\neq0$, weakly in $W^{1,n}\left(
\mathbb{R}^{n}\right) $.
 If
\[
0<p<p_{n}\left(  u\right) :=\frac{1}{\left(  1-\left\Vert u\right\Vert
_{F}^{n}\right)  ^{1/\left(  n-1\right)  }},
\]
then
\[
\underset{k\rightarrow \infty}{\sup}\int_{
\mathbb{R}
^{n}}\Phi\left(  \lambda_{n}p\left\vert u_{k}\right\vert ^{\frac{n}{n-1}
}\right)  dx<+\infty.
\]

\end{lemma}

\section{The maximizing sequence}

Let $\left\{  \beta_{k}\right\}  $ an increasing sequence which
converges to $\lambda_{n}$ and $\left\{  R_{k}\right\}  $ be an increasing sequence which diverges to
infinity as $k\rightarrow \infty$. Denote
\[
I_{\beta_{k}}^{\alpha}\left(  u\right)  =\int_{\mathcal{W}_{R_{k}}}\Phi\left(  \beta
_{k}\left\vert u\right\vert ^{\frac{n}{n-1}}\left(  1+\alpha\left\Vert
u\right\Vert _{{n}}^{{n}}\right)  ^{\frac{1}{n-1}}\right)  dx
\]
and
\[
H=\left\{  \left.  u\in W_{0}^{1,n}\left(  \mathcal{W}_{R_{k}}\right)  \right\vert
\left\Vert u\right\Vert _{F}=1\right\}  .
\]

\begin{lemma}
For any $0\leq\alpha <1$, there exists an extremal function $u_{k}\in H$
such that
\[
I_{\beta_{k}}^{\alpha}\left(  u_{k}\right)  =\underset{u\in H}{\sup}%
I_{\beta_{k}}^{\alpha}\left(  u\right).
\]

\end{lemma}

\begin{proof}
There exists a sequence $\left\{  v_{i}\right\}  \in H$ such that
\[
\underset{i\rightarrow\infty}{\lim}I_{\beta_{k}}^{\alpha}\left(  v_{i}\right)
=\underset{u\in H}{\sup}I_{\beta_{k}}^{\alpha}\left(  u\right)  .
\]
Since $v_{i}$ is bounded in $W^{1,n}\left(
\mathbb{R}
^{n}\right)  $, there exists a subsequence which will still be denoted by $v_{i}$,
such that
\[
\begin{array}
[c]{c}
v_{i}\rightharpoonup u_{k}\text{ weakly in }W^{1,n}\left(
\mathbb{R}
^{n}\right)  ,\\
v_{i}\rightarrow u_{k}\text{ strongly in }L^{s}\left(
\mathcal{W}_{R_{k}}\right) ,\\
v_{i}\rightarrow u_{k}\ \text{ a.e. in}\
\mathbb{R}^{n}
\end{array}
\]
for any $1<s<\infty$ as $i\rightarrow\infty$. Therefore

\begin{align*}
g_{i} &  =\Phi\left\{  \beta_{k}\left\vert v_{i}\right\vert ^{\frac{n}{n-1}%
}\left(  1+\alpha\left\Vert v_{i}\right\Vert _{{n}}^{{n}}\right)  ^{\frac
{1}{n-1}}\right\}  \\
&  \rightarrow g_{k}=\Phi\left\{  \beta_{k}\left\vert u_{k}\right\vert
^{\frac{n}{n-1}}\left(  1+\alpha\left\Vert u_{k}\right\Vert _{{n}}^{{n}%
}\right)  ^{\frac{1}{n-1}}\right\}
\end{align*}
a.e. in $\mathbb{R}^{n}$. Nxet we claim that $u_{k}\neq0$. If not, we have $1+\alpha\left\Vert v_{i}%
\right\Vert _{{n}}^{{n}}\rightarrow1$. Thus $g_{i}$ is bounded in
$L^{r}\left( \mathcal{W}_{R_{k}}\right)  $ for some $r>1$, then $g_{i}\rightarrow0$.
Hence $\underset{u\in H}{\sup}I_{\beta_{k}}^{\alpha}\left(  u\right)
=0$, which is impossible. For any
$p<p_{n}\left(  u_{k}\right)  :=\frac{1}{\left(  1-\left\Vert u_{k}\right\Vert
_{F}^{n}\right)  ^{1/\left(  n-1\right) }}$, it follows from Lemma \ref{jmdo} that
\[
\underset{i\rightarrow\infty}{\lim\sup}\int_{\mathbb{R}^{n}}
\Phi\left(  \lambda_{n}p\left\vert v_{i}\right\vert ^{\frac{n}{n-1}
}\right)  dx<+\infty.
\]
Since $0\leq\alpha < 1$, it is easy to see that
\[
(1+\alpha\left\Vert u_{k}\right\Vert _{{n}}^{{n}})^{\frac{1}{n-1}}<(1+\left\Vert u_{k}\right\Vert
_{F}^{n})^{\frac{1}{n-1}}<\frac{1}{\left(  1-\left\Vert u_{k}\right\Vert
_{F}^{n}\right)  ^{1/\left(  n-1\right) }}
=p_{n}\left(  u_{k}\right)  ,
\]
then $g_{i}$ is bounded in $L^{s}$ for some $s>1$ and $g_{i}\rightarrow
g_{k}$ strongly in $L^{1}\left(  \mathcal{W}_{R_{k}}\right)  $ as $i\rightarrow\infty
$\emph{. }Thus the extremal function is attained for the case $\beta
_{k}<\lambda_{n}$ and $\left\Vert u_{k}\right\Vert _{F}=1$.
\end{proof}
Similar as in \cite{liruf,LuZhu,ZZ2}, we have the following results.
\begin{lemma}
Let $u_{k}$ be as above, then

(i) $u_{k}$ is a maximizing sequence for $S;$

(ii) $u_{k}$ may be chosen to be radially symmetric and decreasing with respect to $F^{0}(x)$.
\end{lemma}

\begin{proof}
(i) Let $\eta$ be a cut-off function which is $1$ on $\mathcal{W}_{1}$ and $0$ on $\mathbb{R}^{n}\backslash \mathcal{W}_{2}$.
Then for any given $\varphi\in W^{1,n}\left(\mathbb{R}^{n}\right)  $ with
$\left\Vert \varphi\right\Vert_{F} =1$, it follows that

\[
\tau^{n}\left( L\right)  :=\int_{
\mathbb{R}^{n}}\left( F^{n}\left(\nabla(\eta( \frac{x}{L}) \varphi)\right)
+\left\vert \eta\left(  \frac{x}{L}\right)  \varphi\right\vert
^{n}\right)  dx\rightarrow1, \ \ \ \ \text{as}\ \ L\rightarrow+\infty.
\]

Thus for a fixed $L$ and $R_{k}>2L$, we have
\begin{align*}
&  \int_{\mathcal{W}_{L}}\Phi\left(  \beta_{k}\left\vert \frac{\varphi}{\tau\left(
L\right)  }\right\vert ^{\frac{n}{n-1}}\left(  1+\alpha\left\Vert \frac
{\eta\left(  \frac{x}{L}\right)  \varphi}{\tau\left(  L\right)  }\right\Vert
_{{n}}^{{n}}\right)  ^{\frac{1}{n-1}}\right)  dx\\
&  \leq\int_{\mathcal{W}_{2L}}\Phi\left(  \beta_{k}\left\vert \frac{\eta\left(  \frac
{x}{L}\right)  \varphi}{\tau\left(  L\right)  }\right\vert ^{\frac{n}{n-1}%
}\left(  1+\alpha\left\Vert \frac{\eta\left(  \frac{x}{L}\right)  \varphi
}{\tau\left(  L\right)  }\right\Vert _{{n}}^{{n}}\right)  ^{\frac{1}{n-1}%
}\right)  dx\\
&  \leq\int_{\mathcal{W}_{R_{k}}}\Phi\left(  \beta_{k}\left\vert u_{k}\right\vert
^{\frac{n}{n-1}}\left(  1+\alpha\left\Vert u_{k}\right\Vert _{{n}}^{{n}%
}\right)  ^{\frac{1}{n-1}}\right)  dx.
\end{align*}
Then it follows from Levi Lemma that

\begin{align*}
& \int_{\mathcal{W}_{L}}\Phi\left(  \lambda_{n}\left\vert \frac{\varphi}{\tau\left(
L\right)  }\right\vert ^{\frac{n}{n-1}}\left(  1+\alpha\left\Vert \frac
{\eta\left(  \frac{x}{L}\right)  \varphi}{\tau\left(  L\right)  }\right\Vert
_{{n}}^{{n}}\right)  ^{\frac{1}{n-1}}\right)  dx\\
\leq&\underset{k\rightarrow\infty}{\lim}\int_{
\mathbb{R}
^{n}}\Phi\left(  \beta_{k}\left\vert u_{k}\right\vert ^{\frac{n}{n-1}}\left(
1+\alpha\left\Vert u_{k}\right\Vert _{{n}}^{{n}}\right)  ^{\frac{1}{n-1}%
}\right)  dx.
\end{align*}

Taking limits \ $L\rightarrow+\infty$,

\[
\int_{\mathbb{R}^{n}}\Phi\left(  \lambda_{n}\left\vert \varphi\right\vert ^{\frac{n}{n-1}%
}\left(  1+\alpha\left\Vert \varphi\right\Vert _{{n}}^{{n}}\right)  ^{\frac
{1}{n-1}}\right)  dx\leq\underset{k\rightarrow\infty}{\lim}\int_{
\mathbb{R}
^{n}}\Phi\left(  \beta_{k}\left\vert u_{k}\right\vert ^{\frac{n}{n-1}}\left(
1+\alpha\left\Vert u_{k}\right\Vert _{{n}}^{{n}}\right)  ^{\frac{1}{n-1}
}\right)  dx.
\]
Thus
\begin{align*}
&  \underset{k\rightarrow\infty}{\lim}\int_{\mathcal{W}_{R_{k}}}\Phi\left(  \beta
_{k}\left\vert u_{k}\right\vert ^{\frac{n}{n-1}}\left(  1+\alpha\left\Vert
u_{k}\right\Vert _{{n}}^{{n}}\right)  ^{\frac{1}{n-1}}\right)  dx\\
=& \underset{\left\Vert u\right\Vert _{F }=1}{\sup}\int_{
\mathbb{R}
^{n}}\Phi\left(  \lambda_{n}\left\vert u\right\vert ^{\frac{n}{n-1}}\left(
1+\alpha\left\Vert u\right\Vert _{{n}}^{{n}}\right)  ^{\frac{1}{n-1}}\right)
dx.
\end{align*}

 Let $u_{k}^{\star}$ be convex symmetric rearrangement of $u_{k}$ with respect to $F^{0}(x)$, then

\[
\tau_{k}^{n}:=\int_{\mathbb{R}^{n}}
\left(  F^{n}(\nabla u_{k}^{\star})+\left\vert
u_{k}^{\star}\right\vert^{n}\right)  dx
\leq\int_{\mathbb{R}^{n}}\left(  F^{n}( \nabla u_{k})+\left\vert u_{k}
\right\vert ^{n}\right)  dx=1.
\]

Therefore

\[
\int_{\mathcal{W}_{R_{k}}}\Phi\left(  \beta_{k}\left\vert \frac{u_{k}^{\star}}{\tau_{k}
}\right\vert ^{\frac{n}{n-1}}\left(  1+\alpha\left\Vert \frac{u_{k}^{\star}
}{\tau_{k}}\right\Vert _{{n}}^{{n}}\right)  ^{\frac{1}{n-1}}\right)
dx\geq\int_{\mathcal{W}_{R_{k}}}\Phi\left(  \beta_{k}\left\vert u_{k}^{\star}\right\vert
^{\frac{n}{n-1}}\left(  1+\alpha\left\Vert u_{k}^{\star}\right\Vert _{{n}}
^{{n}}\right)  ^{\frac{1}{n-1}}\right)  dx.
\]

It is easy to see that
\[
\int_{\mathcal{W}_{R_{k}}}\Phi\left(  \beta_{k}\left\vert u_{k}^{\star}\right\vert
^{\frac{n}{n-1}}\left(  1+\alpha\left\Vert u_{k}^{\star}\right\Vert _{{n}}%
^{{n}}\right)  ^{\frac{1}{n-1}}\right)  dx=\int_{\mathcal{W}_{R_{k}}}\Phi\left(
\beta_{k}\left\vert u_{k}\right\vert ^{\frac{n}{n-1}}\left(
1+\alpha\left\Vert u_{k}\right\Vert _{{n}}^{{n}}\right)  ^{\frac
{1}{n-1}}\right)  dx.
\]

Then one can obtain $\tau_{k}=1$. Also we know the fact that $\tau_{k}=1$ if and only if $u_{k}$ is
radial. Thus

\begin{align*}
&  \int_{\mathcal{W}_{R_{k}}}\Phi\left(  \beta_{k}\left\vert u_{k}^{\star}\right\vert
^{\frac{n}{n-1}}\left(  1+\alpha\left\Vert u_{k}^{\star}\right\Vert _{{n}}%
^{{n}}\right)  ^{\frac{1}{n-1}}\right)  dx\\
=&  \underset{\left\Vert u\right\Vert _{F}%
=1}{\sup}\int_{\mathcal{W}_{R_{k}}}\exp\left\{  \beta_{k}\left\vert u\right\vert
^{\frac{n}{n-1}}\left(  1+\alpha\left\Vert u\right\Vert _{{n}}^{{n}}\right)
^{\frac{1}{n-1}}\right\}  dx.
\end{align*}
Therefore one can assume $u_{k}=u_{k}(r)$ and
$u_{k}\left( r\right)  $ is decreasing with respect to $ r=F^{0}(x)$.
\end{proof}

\section{Proof of Theorem 1.1}

In this section, we will give the proof of Theorem 1.1. Firstly, we prove the second part of Theorem 1.1 by the test functions argument. Then we prove the first part of Theorem 1.1 by considering two cases.
Let $c_{k}=maxu_{k}(x)=u_{k}(0)$. When $\sup_{k}c_{k}<+\infty$, it can be proved by using the concentration-compactness lemma. When  $\sup_{k}c_{k}=+\infty$, we perform a blow-up procedure and prove the corresponding results.

\subsection{Proof of the second part of Theorem 1.1}

In this subsection, we will show that the sumpremum in Theorem \ref{moser-trudinger} is infinity if $\alpha\geq1$. Namely, we prove the sharpness of the inequality in Theorem \ref{moser-trudinger}.  The proof of the second part of Theorem \ref{moser-trudinger} is
based on a test function argument. Unlike in the case for  bounded domains
\cite{Z1}, we cannot construct the test function by the
eigenfunction of the first eigenvalue problem: $$\underset {u\in
W_{0}^{1,n}\left(  \Omega\right)  ,u\neq0}{\inf}\frac{\left\Vert
F(\nabla u)\right\Vert _{n}^{n}}{\left\Vert u\right\Vert _{n}^{n}},$$
since the above infimum is actually not attained when $\Omega=\mathbb{R}
^{n}$.  To overcome this difficulty, we
will construct a new test function sequence.

\begin{proof}
[Proof of the Second Part of Theorem \ref{moser-trudinger}]
Let
\[
u_{k}=\left\{
\begin{array}
[c]{c}
\frac{1}{(n\kappa_{n})^{\frac{1}{n}}}\left(  \log k\right)  ^{\frac{n-1}{n}},\ \ \text{\ \ \ \ \ \ \ \ \ \ \ }
0< F^{0}(x)  \leq\frac{R_{k}}{k},\\
\frac{1}{(n\kappa_{n})^{\frac{1}{n}}}\left(  \log k\right)  ^{-\frac{1}{n}}\ \log\frac{R_{k}}{F^{0}(x)},\text{ \ \ \ \ }\frac{R_{k}}{k}< F^{0}(x) \leq
R_{k},\\
0,\text{ \ \ \ \ \ \ \ \ \ \  \ \ \ \ \ \ \ \ \  \ \ \  \ \ \ \ \  \  \ \ \ } F^{0}(x)>R_{k},
\end{array}
\right.
\]
where $R_{k}:=\frac{\left(  \log k\right)  ^{1/2n}}{\log\log k}\rightarrow
+\infty$ as $k\rightarrow\infty$. It is easy to verify that
\[
\int_{\mathbb{R}
^{n}}F^{n}(\nabla u_{k})dx=1.
\]

Also we have
\begin{align*}
\left\Vert u_{k}\right\Vert _{n}^{n} &  = \int_{\mathcal{W}_{R_{k}}\backslash \mathcal{W}_{R_{k}/k}} \left\vert u_{k}\right\vert
^{n}dx+\int_{\mathcal{W}_{R_{k}/k}}
\left\vert u_{k}\right\vert
^{n}dx\\
&  =\frac{R_{k}^{n}}{\log k}\int_{\frac{1}{k}}^{1}\left(  \log r\right)
^{n}r^{n-1}dr+\frac{\left(  \log k\right)  ^{n-1}}{n}\left(  \frac{R_{k}}{k}\right)
^{n}\\
&  =C_{n}\frac{R_{k}^{n}}{\log k}\left(  1+o\left(  1\right)  \right)
\rightarrow0\text{ as }k\rightarrow\infty,
\end{align*}
where $C_{n}=\int_{\frac{1}{k}}^{1}\left(  \log r\right)  ^{n}r^{n-1}dr$.
Thus

\[
\left\Vert u_{k}\right\Vert _{F}^{n}=1+\frac{C_{n}R_{k}^{n}}{\log k}\left(  1+o\left(  1\right)
\right).
\]
$\ $

Using the following fact
\[
1+\frac{\left\Vert u_{k}\right\Vert _{n}^{n}}{\left\Vert u_{k}\right\Vert
_{F}}=\frac{1+2\left\Vert u_{k}\right\Vert _{n}^{n}}{1+\left\Vert
u_{k}\right\Vert _{n}^{n}},
\]
$\ $
then on the Wulff ball $\mathcal{W}_{R_{k}/k}$, it follows
\begin{align*}
&  \lambda_{n}\frac{\left\vert u_{k}\right\vert ^{\frac{n}{n-1}}}{\left\Vert
u_{k}\right\Vert _{F}^{\frac{n}{n-1}}}\left(  1+\frac{\left\Vert u_{k}\right\Vert
_{n}^{n}}{\left\Vert u_{k}\right\Vert _{F}^{n}}\right)  ^{\frac{1}{n-1}}\\
=& n^{\frac{n}{n-1}}\kappa_{n}^{\frac{1}{n-1}}\left\vert u_{k}\right\vert ^{\frac{n}{n-1}%
}\frac{\left(  1+2\left\Vert u_{k}\right\Vert _{n}^{n}\right)  ^{\frac{1}%
{n-1}}}{\left(  1+\left\Vert u_{k}\right\Vert _{n}^{n}\right)  ^{\frac{2}%
{n-1}}}\\
=& n\log k\left(  1-\frac{1}{n-1}\left\Vert u_{k}\right\Vert _{n}^{2n}%
+\frac{2}{n-1}\left\Vert u_{k}\right\Vert_{n}^{3n}\left(  1+o\left(
1\right)  \right)  \right).
\end{align*}
 Therefore
\begin{align*}
& \underset{\left\Vert u\right\Vert _{F}=1}{\sup}\int_{%
\mathbb{R}
^{n}}\Phi\left( \lambda_{n}\left\vert u\right\vert ^{\frac{n}{n-1}}\left(
1+\left\Vert u\right\Vert _{n}^{n}\right)  \right)  dx\\
\geq& C\int_{\mathcal{W}_{R_{k}/k}}\exp\left(  \lambda_{n}\frac{\left\vert
u_{k}\right\vert ^{\frac{n}{n-1}}}{\left\Vert u_{k}\right\Vert _{F}^{\frac{n}{n-1}}}\left(  1+\frac{\left\Vert u_{k}\right\Vert _{n}^{n}%
}{\left\Vert u_{k}\right\Vert _{F}^{n}}\right)  \right)  dx\\
\geq& C\exp\left(  n\log k\left(  1-\frac{1}{n-1}\left\Vert u_{k}\right\Vert
_{n}^{2n}+\frac{2}{n-1}\left\Vert u_{k}\right\Vert _{n}^{3n}\left(  1+o\left(
1\right)  \right)  \right)  +n\log R_{k}-n\log k\right)  \\
=& C\exp\left(  n\log k\left(  -\frac{1}{n-1}\left\Vert u_{k}\right\Vert
_{n}^{2n}+\frac{2}{n-1}\left\Vert u_{k}\right\Vert _{n}^{3n}\left(  1+o\left(
1\right)  \right)  \right)  +n\log R_{k}\right),
\end{align*}
here we haved use the following result
 $$\left\vert \mathcal{W}_{R_{k}/k}\right\vert =\kappa_{n}(\frac{R_{k}}{k})^{n}=\kappa_{n}\exp\left(
n\log R_{k}-n\log k\right) .$$

Since
\[
n\log R_{k}=n\log\left(  \frac{\left(  \log k\right)  ^{1/2n}}{\log\log
k}\right)  =\frac{1}{2}\log\log k-n\log\log\log k
\]
and$\ $
\begin{align*}
&  n\log k\left(  -\frac{1}{n-1}\left\Vert u_{k}\right\Vert _{n}^{2n}+\frac
{2}{n-1}\left\Vert u_{k}\right\Vert _{n}^{3n}\left(  1+o\left(  1\right)
\right)  \right)  \\
&  =\frac{-n}{n-1}\frac{C_{n}^{2}R_{k}^{2n}}{\log k}\left(  1+o\left(
1\right)  \right)  \\
&  =\frac{-n}{n-1}C_{n}^{2}\frac{1}{\left(  \log\log k\right)  ^{2n}}\left(
1+o\left(  1\right)  \right)  ,
\end{align*}
then
\begin{align*}
& \int_{
\mathbb{R}^{n}}\Phi\left(  \lambda_{n}\left\vert u_{k}\right\vert ^{\frac{n}{n-1}}\left(
1+\left\Vert u_{k}\right\Vert _{n}^{n}\right)  \right)  dx\\
\geq& C\exp\left(  n\log k\left(  -\left\Vert u_{k}\right\Vert _{n}%
^{2n}\left(  1+o\left(  1\right)  \right)  \right)  +n\log R_{k}\right)  \\
=&C\exp\left(  \frac{1}{2}\log\log k-n\log\log\log k-\frac{nC_{n}^{2}}%
{n-1}\frac{1}{\left(  \log\log k\right)  ^{2n}}\left(  1+o\left(  1\right)
\right)  \right)  \\
\rightarrow&  +\infty\ \ \ as\ k\rightarrow\infty.
\end{align*}
The proof of the second part of Theorem \ref{moser-trudinger} has been completed.
\end{proof}
\subsection{Proof of the first part of Theorem 1.1}
\medskip

Now we consider two cases for the proof of the first part of Theorem
\ref{moser-trudinger}. Denote $c_{k}=maxu_{k}(x)=u_{k}(0)$. In the case of $sup_{k}c_{k}<+\infty$, the proof of Theorem \ref{moser-trudinger} is indirect and easy.  In the case of $sup_{k}c_{k}=+\infty$, we will use the blowing up
analysis method, which is based on a blowing up analysis of
sequences of solutions to $n$-anisotropic Laplacian in $\mathbb{R}
^{n}$ with exponential growth. The method has been successfully applied in
the proof of the Moser-Trudinger inequalities and related extremal functions existence
results in bounded domains (see
\cite{Adimurthi,zhu,lu-yang,lu-yang 1}) and in the unbounded domains (see \cite{ruf,LuZhu,ZZ2}).

\subsubsection{Proof in the case of $sup_{k}c_{k}<+\infty$}

By the variational calculation, the Euler-Lagrange equation for the extremal function
$u_{k}\in W_{0}^{1,n}\left(  \mathcal{W}_{R_{k}}\right)  $ of $I_{\beta_{k}}^{\alpha
}\left(  u\right)  $ can be written as
\begin{equation}
-Q_{n}(u_{k})=\mu_{k}\lambda_{k}^{-1}u_{k}^{\frac{1}{n-1}
}\Phi^{\prime}\left\{  \alpha_{k}u_{k}^{\frac{n}{n-1}}\right\}
+(\gamma_{k}-1)u_{k}^{{n}-1}, \label{equation}
\end{equation}
where
\[
\left\{
\begin{array}
[c]{c}
u_{k}\in W_{0}^{1,n}\left( \mathcal{W}_{R_{k}}\right)  ,\left\Vert u_{k}\right\Vert
_{F}=1,\\
\alpha_{k}=\beta_{k}\left(  1+\alpha\left\Vert u_{k}\right\Vert _{{n}}^{{n}%
}\right)  ^{\frac{1}{n-1}},\\
\mu_{k}=\left(  1+\alpha\left\Vert u_{k}\right\Vert _{{n}}^{{n}}\right)
/\left(  1+2\alpha\left\Vert u_{k}\right\Vert _{{n}}^{{n}}\right)  ,\\
\gamma_{k}=\alpha/\left(  1+2\alpha\left\Vert u_{k}\right\Vert _{{n}}^{{n}%
}\right)  ,\\
\lambda_{k}=\int_{\mathcal{W}_{R_{k}}}u_{k}^{\frac{n}{n-1}}\Phi^{\prime}\left(
\alpha_{k}u_{k}^{\frac{n}{n-1}}\right) dx .
\end{array}
\right.
\]

 Let us give the following important results firstly.

\begin{lemma}\label{lem4.1}
\label{lamna}$\underset{k}{\inf}\, \,\lambda_{k}>0.$
\end{lemma}

\begin{proof}
We prove this result by contradiction. Assume $\lambda_{k}\rightarrow0$ as $k\rightarrow\infty$, then
\begin{align}
\lambda_{k}  &  =\int_{
\mathbb{R}
^{n}}u_{k}^{\frac{n}{n-1}}\Phi^{\prime}\left(  \alpha_{k}u_{k}^{\frac{n}{n-1}
}\right)  dx=\int_{
\mathbb{R}
^{n}}u_{k}^{\frac{n}{n-1}}\underset{j=n-2}{\overset{\infty}{\sum}}
\frac{\left(  \alpha_{k}u_{k}^{\frac{n}{n-1}}\right)  ^{j}}{j!}dx\nonumber\\
&  =\int_{
\mathbb{R}
^{n}}\left(  \frac{\alpha_{k}^{n-2}u_{k}^{n}}{\left(  n-2\right)  !}
+\cdots\right)  dx\geq\frac{\alpha_{k}^{n-2}}{\left(  n-2\right)  !}\int_{
\mathbb{R}
^{n}}u_{k}^{n}dx. \label{1}
\end{align}
Since $u_{k}(r) $ is decreasing, we have
$u_{k}^{n}\left(  L\right)  \left\vert \mathcal{W}_{L}\right\vert \leq\int_{\mathcal{W}_{L}}%
u_{k}^{n}dx\leq1$, thus
\begin{equation}
u_{k}^{n}\left(  L\right)  \leq\frac{1}{\kappa_{n}L^{n}}. \label{2}%
\end{equation}
Set $\varepsilon^{n}=\frac{1}{\kappa_{n}L^{n}}$, then we
get $u_{k}\leq\varepsilon$ for any $x\notin \mathcal{W}_{L}$. Thus
\[
\int_{
\mathbb{R}
^{n}\backslash \mathcal{W}_{L}}\Phi\left(  \alpha_{k}u_{k}^{\frac{n}{n-1}}\right)
dx\leq C\int_{
\mathbb{R}
^{n}\backslash \mathcal{W}_{L}}u_{k}^{n}dx\leq C\lambda_{k}\rightarrow0.
\]
It is easy to see that
\[
\Phi\left(  \alpha_{k}u_{k}^{\frac{n}{n-1}}\right)  =\underset{j=n-1}%
{\overset{\infty}{\sum}}\frac{\left(  \alpha_{k}u_{k}^{\frac{n}{n-1}}\right)
^{j}}{j!}\leq\underset{j=n-2}{\overset{\infty}{\sum}}\frac{\alpha_{k}%
u_{k}^{\frac{n}{n-1}}\left(  \alpha_{k}u_{k}^{\frac{n}{n-1}}\right)  ^{j}%
}{\left(  j+1\right)  j!}\leq \alpha_{k}u_{k}^{\frac{n}{n-1}}\Phi^{\prime}\left(
\alpha_{k}u_{k}^{\frac{n}{n-1}}\right) ,
\]
then we obtain
\begin{align*}
&  \underset{k\rightarrow\infty}{\lim}\int_{\mathcal{W}_{L}}\Phi\left(  \alpha_{k}
u_{k}^{\frac{n}{n-1}}\right)  dx\\
=&\underset{k\rightarrow\infty}{\lim}\left(
\int_{\mathcal{W}_{L}\cap\left\{  u_{k}\geq1\right\}  }+\int_{\mathcal{W}_{L}\cap\left\{
u_{k}<1\right\}  }\right)  \Phi\left(  \alpha_{k}u_{k}^{\frac{n}{n-1}}\right)
dx\\
\leq& \underset{k\rightarrow\infty}{\lim}\left( C\int_{\mathcal{W}_{L}}u_{k}^{\frac
{n}{n-1}}\Phi\left(  \alpha_{k}u_{k}^{\frac{n}{n-1}}\right)  dx+\int
_{\mathcal{W}_{L}\cap\left\{  u_{k}<1\right\}  }\Phi\left(  \alpha_{k}u_{k}^{\frac
{n}{n-1}}\right)  dx\right)  \\
\leq& C\underset{k\rightarrow\infty}{\lim}\left( \lambda_{k}+\int_{\mathcal{W}_{L}
}u_{k}^{n}dx\right)  .
\end{align*}
By (\ref{1}), it follows that $\int_{\mathcal{W}_{L}}u_{k}^{q}dx\rightarrow0$. Thus for any
$q>1$, we obtain

\[
\underset{k\rightarrow\infty}{\lim}\int_{\mathcal{W}_{L}}\Phi\left(  \alpha_{k}%
u_{k}^{\frac{n}{n-1}}\right)  dx=0.
\]
It is impossible and the proof of Lemma \ref{lem4.1} is finished.
\end{proof}

 Now we recall the concept of Sobolev-normalized concentrating
sequence and concentration-compactness principle as in  \cite{ruf}.

\begin{definition}
\label{SNC}A sequence $\left\{  u_{k}\right\}  \in W^{1,n}\left(
\mathbb{R}
^{n}\right)  $ is a Sobolev-normalized concentrating sequence, if

i) $\left\Vert u_{k}\right\Vert _{F}=1;$

ii) $u_{k}\rightharpoonup0$ weakly in $W^{1,n}\left(\mathbb{R}^{n}\right) ;$

iii) there exists a point $x_{0}$ such that for any $\delta>0$, $\int_{
\mathbb{R}
^{n}\backslash \mathcal{W}_{\delta}\left(  x_{0}\right)  }\left( F^{n}
(\nabla u_{k})+\left\vert u_{k}\right\vert ^{n}\right)
dx\rightarrow0$.
\end{definition}

From Lemma \ref{jmdo} in this paper, we have the following result.
\begin{lemma}
\label{lions}Let $\left\{  u_{k}\right\}  $ be a sequence satisfying
$\left\Vert u_{k}\right\Vert _{F}=1$, and $u_{k}\rightharpoonup u$ weakly in
$W^{1,n}\left(\mathbb{R}^{n}\right) $.
Then either $\left\{  u_{k}\right\}  $ is a
Sobolev-normalized concentrating sequence, or there exists
$\gamma>0$ such that $\Phi\left(
\left(  \lambda_{n}+\gamma\right)  \left\vert u_{k}\right\vert ^{\frac{n}{n-1}%
}\right)  $ is bounded in $L^{1}\left(\mathbb{R}
^{n}\right)  $. \end{lemma}

\begin{theorem}
\label{attain lemma2}\bigskip If $\ \underset{k}{\sup}c_{k}<+\infty$,
then the first part of Theorem \ref{moser-trudinger} holds.
\end{theorem}

\begin{proof}
For any $\varepsilon>0$, by using (\ref{2}), there exist some $L$ such that
$u_{k}\left(  x\right)  \leq\varepsilon$ when $x\notin \mathcal{W}_{L}$. It is easy to see that
\[\int_{
\mathbb{R}
^{n}}\left(  \Phi\left(  \alpha_{k}u_{k}^{\frac{n}{n-1}}\right)  -\frac
{\alpha_{k}^{n-1}u_{k}^{n}}{\left(  n-1\right)  !}\right)  dx
=\left(  \int_{\mathcal{W}_{L}}+\int_{
\mathbb{R}
^{n}\backslash\mathcal{W}_{L}}\right)  \left(  \Phi\left(  \alpha_{k}u_{k}^{\frac
{n}{n-1}}\right)  -\frac{\alpha_{k}^{n-1}u_{k}^{n}}{\left(  n-1\right)
!}\right)  dx.
\]
Also
\[
\int_{
\mathbb{R}
^{n}\backslash \mathcal{W}_{L}}\left(  \Phi\left(  \alpha_{k}u_{k}^{\frac{n}{n-1}
}\right)  -\frac{\alpha_{k}^{n-1}u_{k}^{n}}{\left(  n-1\right)  !}\right)
dx=C\int_{
\mathbb{R}
^{n}\backslash \mathcal{W}_{L}}u_{k}^{\frac{n^{2}}{n-1}}dx\leq C\varepsilon^{\frac
{n^{2}}{n-1}-n}\int_{
\mathbb{R}
^{n}}u_{k}^{n}dx=C\varepsilon^{\frac{n^{2}}{n-1}-n}.
\]
Then
\begin{equation}
\int_{
\mathbb{R}
^{n}}\left(  \Phi\left(  \alpha_{k}u_{k}^{\frac{n}{n-1}}\right)  -\frac
{\alpha_{k}^{n-1}u_{k}^{n}}{\left(  n-1\right)  !}\right)  dx=\int_{\mathcal{W}_{L}
}\left(  \Phi\left(  \alpha_{k}u_{k}^{\frac{n}{n-1}}\right)  -\frac{\alpha
_{k}^{n-1}u_{k}^{n}}{\left(  n-1\right)  !}\right)  dx+O\left(  \varepsilon
^{\frac{n^{2}}{n-1}-n}\right).  \label{33}
\end{equation}
By $\underset{k}{\sup}c_{k}<+\infty$, we have
\begin{align*}
\int_{
\mathbb{R}
^{n}}\Phi\left(  \alpha_{k}u_{k}^{\frac{n}{n-1}}\right)  dx  &  =\int_{\mathcal{W}_{L}%
}\left(  \Phi\left(  \alpha_{k}u_{k}^{\frac{n}{n-1}}\right)  -\frac{\alpha
_{k}^{n-1}u_{k}^{n}}{\left(  n-1\right)  !}\right)  dx+\int_{%
\mathbb{R}
^{n}}\frac{\alpha_{k}^{n-1}u_{k}^{n}}{\left(  n-1\right)  !}dx+O\left(
\varepsilon^{\frac{n^{2}}{n-1}-n}\right) \\
&  \leq C\left(  L\right).
\end{align*}
Thus the proof of Theorem \ref{moser-trudinger} in the case of $\underset{k}{\sup}c_{k}<+\infty$ is finished.
\end{proof}

\subsubsection{Blow-up analysis and proof in the case of $\sup_{k}c_{k}=+\infty$}
In the following, we assume $\sup_{k}c_{k}=+\infty$ and perform a blow-up procedure. The method of blow-up analysis will be used to analyze the
asymptotic behavior of the maximizing sequence $\left\{  u_{k}\right\} $, and
the first part of Theorem \ref{moser-trudinger} in the case of $\sup_{k}c_{k}=+\infty$ will be proved.

First, we denote

\[
r_{k}^{n}=\frac{\lambda_{k}}{\mu_{k}c_{k}^{\frac{n}{n-1}}e^{\alpha_{k}%
c_{k}^{\frac{n}{n-1}}}}.
\]
By (\ref{2}), one can find a sufficiently large $L$ such that $u_{k}\leq1$ on $%
\mathbb{R}
^{n}\backslash \mathcal{W}_{L}\,$. Then $\left(  u_{k}-u_{k}\left(  L\right)  \right)
^{+}\in W_{0}^{1,n}\left( \mathcal{W}_{L}\right)  $ and

\[
\int_{\mathcal{W}_{L}}F^{n}( \nabla\left(  u_{k}-u_{k}\left(  L\right)  \right)
^{+})dx\leq1.
\]
By Theorem 1.1 in \cite{Z1}, we know that if \[
\beta<\underset{u\in W_{0}^{1,n}\left(  \mathcal{W}_{L}\right)  }{\inf}\frac{\left\Vert F(
\nabla u)\right\Vert_{n}^{n}}{\left\Vert u\right\Vert _{n}^{n}},
\]
then

\[
\int_{\mathcal{W}_{L}}\exp\{{\lambda_{n}\left(  u_{k}
-u_{k}\left(  L\right)  \right)  ^{\frac{n}{n-1}}\left(  1+\beta\left\Vert u_{k}-u_{k}\left(
L\right)  \right\Vert _{n}^{n}\right)  ^{\frac{1}{n-1}}}\}dx\leq C\left(  L\right).
\]

For any $q<\lambda_{n}\left(  1+\beta\left\Vert u_{k}-u_{k}\left(  L\right)
\right\Vert _{n}^{n}\right)  ^{\frac{1}{n-1}}$, there exists a constant
$C\left(  q\right)  \,\ $such that

\[
qu_{k}^{\frac{n}{n-1}}\leq\lambda_{n}\left(
\left(  u_{k}-u_{k}\left(  L\right)  \right)  ^{+}\right)  ^{\frac{n}{n-1}}
\left(  1+\beta\left\Vert u_{k}
-u_{k}\left(  L\right)  \right\Vert _{n}^{n}\right)  ^{\frac{1}{n-1}}
+C\left(  q\right).
\]
Then

\begin{equation}
\int_{\mathcal{W}_{L}}\exp\{{qu_{k}^{\frac{n}{n-1}}}\}dx\leq C\left(  L,q\right)  .
\label{add7}%
\end{equation}
Taking some $0<A<1$ such that
\[
\left(  1-A\right)  \beta_{k}\left(  1+\alpha\left\Vert u_{k}\right\Vert
_{{n}}^{{n}}\right)  ^{\frac{1}{n-1}}<\lambda_{n}\left(  1+\beta\left\Vert
u_{k}-u_{k}\left(  L\right)  \right\Vert _{n}^{n}\right)  ^{\frac{1}{n-1}},
\]
then
\begin{align*}
&  \lambda_{k}e^{-A\beta_{k}\left(  1+\alpha\left\Vert u_{k}\right\Vert _{{n}
}^{{n}}\right)  ^{\frac{1}{n-1}}c_{k}^{\frac{n}{n-1}}}\\
&  =e^{-A\beta_{k}\left(  1+\alpha\left\Vert u_{k}\right\Vert _{{n}}^{{n}
}\right)  ^{\frac{1}{n-1}}c_{k}^{\frac{n}{n-1}}}\left[  \left(  \int_{
\mathbb{R}
^{n}\backslash \mathcal{W}_{L}}+\int_{\mathcal{W}_{L}}\right)  u_{k}^{\frac{n}{n-1}}\Phi^{\prime
}\left(  \alpha_{k}u_{k}^{\frac{n}{n-1}}\right)  dx\right]  \\
&  \leq Ce^{-A\beta_{k}\left(  1+\alpha\left\Vert u_{k}\right\Vert _{{n}}
^{{n}}\right)  ^{\frac{1}{n-1}}c_{k}^{\frac{n}{n-1}}}\left(  \int_{
\mathbb{R}
^{n}\backslash \mathcal{W}_{L}}u_{k}^{n}dx+\int_{ \mathcal{W}_{L}}u_{k}^{\frac{n}{n-1}}e^{\beta
_{k}\left(  1+\alpha\left\Vert u_{k}\right\Vert _{{n}}^{{n}}\right)
^{\frac{1}{n-1}}u_{k}^{\frac{n}{n-1}}}dx\right)  \\
&  \leq C\int_{ \mathcal{W}_{L}}u_{k}^{\frac{n}{n-1}}e^{\left(  1-A\right)  \beta
_{k}\left(  1+\alpha\left\Vert u_{k}\right\Vert _{{n}}^{{n}}\right)
^{\frac{1}{n-1}}u_{k}^{\frac{n}{n-1}}}dx+o\left(  1\right)  .
\end{align*}
Since $u_{k}$ converges strongly in $L^{s}\left( \mathcal{W}_{L}\right)  $ for any
$s>1$, by (\ref{add7}), it follows that

\[
\lambda_{k}\leq Ce^{A\alpha_{k}c_{k}^{\frac{n}{n-1}}}.
\]
Then for any \thinspace$q>0$, we have
\begin{equation}
\ r_{k}^{n}\leq Ce^{\left(  A-1\right)  \alpha_{k}c_{k}^{\frac{n}{n-1}}%
}=o\left(  c_{k}^{-q}\right)  . \label{rk}
\end{equation}

\bigskip

Set
\[
\left\{
\begin{array}
[c]{c}
m_{k}\left(  x\right)  =u_{k}\left(  r_{k}x\right)  ,\\
\phi_{k}\left(  x\right)  =\frac{m_{k}\left(  x\right)  }{c_{k}},\\
\psi_{k}\left(  x\right)  =\frac{n}{n-1}\alpha_{k}c_{k}^{\frac{1}{n-1}}\left(
m_{k}-c_{k}\right)  ,
\end{array}
\right.
\]
where $m_{k},\phi_{k}$ and $\psi_{k}$ are defined on $\Omega_{k}:=\left\{
x\in\mathbb{R}^{n}:r_{k}x\in \mathcal{W}_{1}\right\}  $. From (\ref{equation}) and (\ref{rk}), it is easy to see that
$\phi_{k}\left(  x\right)$  and $\psi_{k}\left(  x\right) $ respectively satisfy

\begin{align}
-Q_{n}\phi_{k}\left(  x\right)   &  =\frac{r_{k}^{n}}{c_{k}^{n-1}%
}\left(  \mu_{k}\lambda_{k}^{-1}m_{k}^{\frac{1}{n-1}}\Phi^{\prime}\left\{
\alpha_{k}m_{k}^{\frac{n}{n-1}}\right\}  +\left(  \gamma_{k}-1\right)
m_{k}^{{n}-1}\right) \nonumber\\
&  =\left(  \frac{1}{c_{k}^{n}}\phi_{k}^{\frac{1}{n-1}}\left(  x\right)
\Phi^{\prime}\left\{  \alpha_{k}\left(  m_{k}^{\frac{n}{n-1}}-c_{k}^{\frac
{n}{n-1}}\right)  \right\}  +o\left(  1\right)  \right)  \label{equ 1},
\end{align}

\begin{align}
-Q_{n}\psi_{k}\left(  x\right)   &  =\left(  \frac{n\alpha_{k}}%
{n-1}\right)  ^{n-1}c_{k}r_{k}^{n}\left(  \mu_{k}\lambda_{k}^{-1}m_{k}%
^{\frac{1}{n-1}}\Phi^{\prime}\left\{  \alpha_{k}m_{k}^{\frac{n}{n-1}}\right\}
+\left(  \gamma_{k}-1\right)  m_{k}^{n-1}\right) \nonumber\\
&  =\left(  \frac{n\alpha_{k}}{n-1}\right)  ^{n-1}\left(  \left(  \frac{m_{k}%
}{c_{k}}\right)  ^{\frac{1}{n-1}}e^{\alpha_{k}\left(  m_{k}^{\frac{n}{n-1}%
}-c_{k}^{\frac{n}{n-1}}\right)  }+o\left(  1\right)  \right)  . \label{equ 2}%
\end{align}

Now let us analyze the limit function of $\phi_{k}(x)$ and $\psi_{k}(x)$.
Because $u_{k}$ is bounded in $W^{1,n}\left(
\mathbb{R}
^{n}\right)  $, there exists a subsequence such that $u_{k}\rightharpoonup u$
weakly in $W^{1,n}\left(
\mathbb{R}
^{n}\right)  $. Since the right side of (\ref{equ 1}) vanishes as
$k\rightarrow\infty$, then $\phi_{k}\rightarrow\phi$ in $C_{loc}
^{1}\left(
\mathbb{R}
^{n}\right)  $ as $k\rightarrow\infty$, by applying the classical eatimates
\cite{Tolksdorf}, we have

\[
-Q_{n}\phi(x)=0\text{ in }
\mathbb{R}
^{n}.\]
Since $\phi_{k}\left(  0\right)  =1$, Liouville type theorem (see \cite{HKM}) asserts that
$\phi\equiv1$ in $
\mathbb{R}
^{n}$.

\medskip

Now we  analyze the asymptotic behavior of $\psi_{k}$. By (\ref{rk}) and $\phi_{k}\left(  x\right)  \leq1$, we can rewrite (\ref{equ 2})
as
\[
-Q_{n}\psi_{k}\left(  x\right)  =O\left(  1\right)  .
\]
By Theorem 7 in \cite{S4}, we have $osc_{\mathcal{W}_{L}}\psi_{k}\leq C\left(
L\right)  $ for any $L>0.$ Then from the result of \cite{Tolksdorf}, one can get
$\left\Vert \psi_{k}\right\Vert _{C^{1,\delta}\left(  \mathcal{W}_{L}\right)  }\leq
C\left(  L\right)  $ for some $\delta >0$. Thus $\psi_{k}$ converges in $C_{loc}^{1}\left(
\mathcal{W}_{L}\right)  $ and $m_{k}-c_{k}\rightarrow0$ in $C_{loc}^{1}\left(
\mathcal{W}_{L}\right)  $.

It is easy to see that
\[
m_{k}^{\frac{n}{n-1}}=c_{k}^{\frac{n}{n-1}}\left(  1+\frac{m_{k}-c_{k}}{c_{k}
}\right)  ^{\frac{n}{n-1}}=c_{k}^{\frac{n}{n-1}}\left(  1+\frac{n}{n-1}
\frac{m_{k}-c_{k}}{c_{k}}+O\left(  \frac{1}{c_{k}^{2}}\right)  \right) ,
\]
then
\begin{align}
\alpha_{k}\left(  m_{k}^{\frac{n}{n-1}}-c_{k}^{\frac{n}{n-1}}\right)
&=\alpha_{k}c_{k}^{\frac{n}{n-1}}\left(  \frac{n}{n-1}\frac{m_{k}-c_{k}}{c_{k}
}+O\left(  \frac{1}{c_{k}^{2}}\right)  \right) \label{5}\\
&  =\psi_{k}\left(  x\right)  +o\left(  1\right)  \rightarrow\psi\left(
x\right)  \text{ in }C_{loc}^{0}(\mathbb{R}^{n}).\nonumber
\end{align}
Thus
\begin{equation}
-Q_{n}\psi(x)=(  \frac{nc_{n}}{n-1})  ^{n-1}e^{
\psi(x)}  , \label{6}
\end{equation}
where $c_{n}=\underset{k\rightarrow\infty}{\lim}\alpha_{k}=\lambda_{n}(
1+\alpha\underset{k\rightarrow\infty}{\lim}\left\Vert u_{k}\right\Vert _{{n}
}^{{n}})  ^{\frac{1}{n-1}}$.

Since $\psi$ is radially symmetric and decreasing, we know that
(\ref{6}) has only one solution. Thus we have
\[
\psi\left(  x\right)  =-n\log\left(  1+\frac{c_{n}}{n^{\frac{n}{n-1}}
} F^{0}(x) ^{\frac{n}{n-1}}\right) .
\]
Therefore
\begin{align}
\int_{
\mathbb{R}
^{n}}e^{\psi\left(  x\right)  }dx  &  =(n-1)\kappa_{n}\left(
\frac{n^{\frac{n}{n-1}}}{c_{n}}\right)  ^{n-1}\int_{0}^{\infty}\left(
1+t\right)  ^{-n}t^{n-2}dt\nonumber\\
&  =(n-1)\kappa_{n}\left(  \frac{n^{\frac{n}{n-1}}}{c_{n}}\right)
^{n-1}\cdot\frac{1}{n-1}=\frac{1}{1+\alpha\underset{k\rightarrow\infty}{\lim
}\left\Vert u_{k}\right\Vert _{{n}}^{{n}}}. \label{add3}
\end{align}

For any $A>1$, denote $u_{k}^{A}=\min\left\{  u_{k},\frac{c_{k}}{A}\right\} $.

\begin{lemma}\label{lem4.4}
\bigskip For any $A>1$, we have
\[
\underset{k\rightarrow\infty}{\lim\sup}\int_{
\mathbb{R}
^{n}}\left(F^{n}\left (\nabla u_{k}
^{A}\right)+ \left\vert u_{k}^{A}\right\vert ^{n}\right)  dx\leq1-\frac{A-1}{A}\frac{1}{1+\alpha
\underset{k\rightarrow\infty}{\lim}\left\Vert u_{k}\right\Vert _{{n}}^{{n}}}.
\]

\end{lemma}

\begin{proof}
\bigskip Since $\left\vert \left\{  x:u_{k}\geq\frac{c_{k}}{A}\right\}
\right\vert \left\vert \frac{c_{k}}{A}\right\vert ^{n}\leq\int_{\left\{
u_{k}\geq\frac{c_{k}}{A}\right\}  }\left\vert u_{k}\right\vert ^{n}dx\leq1$,
there exises a sequence $\rho_{k}\rightarrow0$ such that
\[
\left\{  x:u_{k}\geq\frac{c_{k}}{A}\right\}  \subset \mathcal{W}_{\rho_{k}}.
\]
Since for any $s>1$, $u_{k}$ converges in $L^{s}\left(  \mathcal{W}_{1}\right)  $, then we
obtain
\[
\underset{k\rightarrow\infty}{\lim}\int_{\left\{  u_{k}\geq\frac{c_{k}}%
{A}\right\}  }\left\vert u_{k}^{A}\right\vert ^{s}dx\leq\underset
{k\rightarrow\infty}{\lim}\int_{\left\{  u_{k}\geq\frac{c_{k}}{A}\right\}
}\left\vert u_{k}\right\vert ^{s}dx=0.
\]
Thus for any $s>0$, it follows
\[
\underset{k\rightarrow\infty}{\lim}  \int_{
\mathbb{R}
^{n}}\left(  u_{k}-\frac{c_{k}}{A}\right)  ^{+}\left\vert u_{k}\right\vert
^{s}dx=0.
\]

Testing (\ref{equation}) with $\left(  u_{k}-\frac{c_{k}}{A}\right)  ^{+}$, we obtain
\begin{align*}
&  \int_{
\mathbb{R}
^{n}}\left( F^{n}\left(\nabla\left(  u_{k}-\frac{c_{k}}{A}\right)
^{+}\right)+\left(  u_{k}-\frac{c_{k}}{A}\right)  ^{+}\left\vert
u_{k}\right\vert ^{n-1}\right)  dx\\
&  =\int_{\mathbb{R}^{n}}\left(  u_{k}-\frac{c_{k}}{A}\right)  ^{+}\mu_{k}\lambda_{k}^{-1}%
u_{k}^{\frac{1}{n-1}}\Phi^{\prime}\left\{  \alpha_{k}u_{k}^{\frac{n}{n-1}%
}\right\}  dx+o\left(  1\right)  \\
&  \geq\int_{\mathcal{W}_{Rr_{k}}}\left(  u_{k}-\frac{c_{k}}{A}\right)  ^{+}\mu
_{k}\lambda_{k}^{-1}u_{k}^{\frac{1}{n-1}}\exp\left\{  \alpha_{k}u_{k}%
^{\frac{n}{n-1}}\right\}  dx+o\left(  1\right)  \\
&  =\int_{\mathcal{W}_{R}}\frac{\left(  m_{k}-\frac{c_{k}}{A}\right)  }{c_{k}}\left(
\frac{m_{k}-c_{k}}{c_{k}}+1\right)  ^{\frac{1}{n-1}}\exp\left\{  \psi
_{k}\left(  x\right)  +o\left(  1\right)  \right\}  dx+o\left(  1\right)  \\
&  \geq\frac{A-1}{A}\int_{\mathcal{W}_{R}}e^{\psi\left(  x\right)  }dx.
\end{align*}
Taking limits $R\rightarrow\infty$ and $k\rightarrow\infty$, then it follows from (\ref{add3}) that
\[
\underset{k\rightarrow\infty}{\lim\inf}\int_{
\mathbb{R}
^{n}}\left( F^{n}\left(\nabla\left(  u_{k}-\frac{c_{k}}{A}\right)
^{+}\right)+\left(  u_{k}-\frac{c_{k}}{A}\right)  ^{+}\left\vert
u_{k}\right\vert ^{n-1}\right)  dx\geq\frac{A-1}{A}\frac{1}{\left(
1+\alpha\underset{k\rightarrow\infty}{\lim}\left\Vert u_{k}\right\Vert _{{n}
}^{{n}}\right)  }.
\]
Hence
\begin{align*}
&  \int_{
\mathbb{R}
^{n}}\left( F^{n}(\nabla u_{k}^{A})+\left\vert u_{k}
^{A}\right\vert ^{n}\right)  dx\\
=&1-\int_{
\mathbb{R}
^{n}}\left(  F^{n}\left(\nabla\left(  u_{k}-\frac{c_{k}}{A}\right)
^{+}\right)+\left(  u_{k}-\frac{c_{k}}{A}\right)  ^{+}\left\vert
u_{k}\right\vert ^{n-1}\right)  dx\\
& \ \ +\int_{
\mathbb{R}
^{n}}\left(  u_{k}-\frac{c_{k}}{A}\right)  ^{+}\left\vert u_{k}\right\vert
^{n-1}dx-\int_{\left\{  u_{k}>\frac{c_{k}}{A}\right\}  }\left\vert
u_{k}\right\vert ^{n}dx+\int_{\left\{  u_{k}>\frac{c_{k}}{A}\right\}
}\left\vert u_{k}^{A}\right\vert ^{n}dx\\
\leq& 1-\frac{A-1}{A}\frac{1}{1+\alpha\underset{k\rightarrow\infty}{\lim
}\left\Vert u_{k}\right\Vert _{{n}}^{{n}}}+o\left(  1\right).
\end{align*}
Then the proof of Lemma \ref{lem4.4} is completed.
\end{proof}

\begin{lemma}\label{lem4.5}
$\underset{k\rightarrow\infty}{\lim}\left\Vert u_{k}\right\Vert _{{n}}^{{n}}=0$.
\end{lemma}

\begin{proof}
If $\{u_{k}\}$ is a Sobolev-normalized concentrating sequence, then $\underset{k\rightarrow\infty}{\lim}\left\Vert u_{k}\right\Vert _{{n}}^{{n}}=0$.
If $\{u_{k}\}$ is not a Sobolev-normalized concentrating sequence, and $\underset{k\rightarrow\infty}{\lim}\left\Vert u_{k}\right\Vert _{{n}}^{{n}}\neq0$. For  $A$ large enough, there exist
some constant $\varepsilon_{0}>0$ such that
\[
\int_{%
\mathbb{R}
^{n}}\left( F^{n}(\nabla u_{k}^{A})+\left\vert u_{k}%
^{A}\right\vert ^{n}\right)  dx=1-\frac{1}{  1+\left(  \alpha
+\varepsilon_{0}\right)  \underset{k\rightarrow\infty}{\lim}\left\Vert u_{k}\right\Vert _{{n}%
}^{{n}}  }<1.
\]
By Theorem 1.2 in \cite{ZZ2}, we have $\int_{
\mathbb{R}
^{n}}\Phi(  q\lambda_{n}\vert u_{k}^{A}\vert ^{\frac{n}{n-1}
})  dx<+\infty$, provided
$$q< \left( \frac{1+\left(
\alpha+\varepsilon_{0}\right)  \underset{k\rightarrow\infty}{\lim}\left\Vert u_{k}\right\Vert
_{{n}}^{{n}}}{\left(
\alpha+\varepsilon_{0}\right)  \underset{k\rightarrow\infty}{\lim}\left\Vert u_{k}\right\Vert
_{{n}}^{{n}}}\right)^{\frac{1}{n-1}}.$$

Since $\alpha<1$, $\left\Vert u_{k}\right\Vert _{F}=1$ and
$\underset{k\rightarrow\infty}{\lim}\left\Vert u_{k}\right\Vert _{{n}}^{{n}}\neq0$, one can take
some $\varepsilon_{0}$ such that $\left(  \alpha+\varepsilon_{0}\right)
\underset{k\rightarrow\infty}{\lim}\left\Vert u_{k}\right\Vert _{{n}}^{{n}}<1$, thus
\[
\left(  1+\alpha\underset{k\rightarrow\infty}{\lim}\left\Vert u_{k}\right\Vert _{{n}}^{{n}%
}\right)  ^{\frac{1}{n-1}}<\left(  \frac{1+\left(  \alpha+\varepsilon
_{0}\right)  \underset{k\rightarrow\infty}{\lim}\left\Vert u_{k}\right\Vert _{{n}}^{{n}}%
}{\left(  \alpha+\varepsilon_{0}\right)  \underset{k\rightarrow\infty}{\lim}\left\Vert
u_{k}\right\Vert _{{n}}^{{n}}}\right)  ^{\frac{1}{n-1}}.
\]
Thus for some $t>1$, we have
\begin{equation}
\int_{
\mathbb{R}
^{n}}\Phi\left(  t\alpha_{k}\left\vert u_{k}^{A}\right\vert
^{\frac{n}{n-1}}\right)  dx<+\infty.\label{add}
\end{equation}

Next we claim that $Q_{n}u_{k}\in L^{r} $ for some $r>1$. When $\int_{\left\{  u_{k}>\frac{c_{k}}{A}\right\}  }F^{n}(\nabla
u_{k})dx\rightarrow0$ as $k\rightarrow\infty$, it is easy to prove the above claim by (\ref{add}) and the classical anisotropic Moser-Trudinger inequalities on bounded domains. When $\int_{\left\{  u_{k}>\frac{c_{k}}{A}\right\} }
F^{n}(\nabla
u_{k})dx\geq C$ for some $C>0$, we split $u_{k}$ as
$u_{k}^{1}+u_{k}^{2}$ with $u_{k}^{1}\rightarrow C\delta_{0}$ and
$\int_{\left\{  u_{k}>\frac{c_{k}}{A}\right\}  }F^{n}(\nabla u_{k}
^{2})dx\rightarrow0$.\ \ \  Then it follows from $\alpha<1$ that
\begin{align*}
1+\alpha\left\Vert u_{k}\right\Vert _{n}^{n}  & =1+\alpha\left\Vert u_{k}
^{2}\right\Vert _{n}^{n}+o_{k}\left(  1\right)  <\frac{1}{1-\left\Vert
u_{k}^{2}\right\Vert _{F}^{n}}+o_{k}\left(  1\right)  \\
& \leq\frac{1}{\left\Vert F(\nabla u_{k}^{1})\right\Vert _{n}^{n}}+o_{k}\left(
1\right)  \leq\frac{1}{\left\Vert F( \nabla u_{k})\right\Vert _{L^{n}\left(
\left\{  u_{k}>\frac{c_{k}}{A}\right\}  \right)  }^{n}}+o_{k}\left(  1\right)
.
\end{align*}
Thus there exist some constant $s>1$ such that $\left(  1+\alpha\left\Vert u_{k}\right\Vert _{n}^{n}\right)  s$ $\leq\frac
{1}{\left\Vert F(\nabla u_{k})\right\Vert _{L^{n}\left(  \left\{  u_{k}%
>\frac{c_{k}}{A}\right\}  \right)  }^{n}}$. Therefore by (\ref{add}) and the classical anisotropic
Moser-Trudinger inequality on the bounded domain, the claim is proved.

Based on the claim above and Lemma \ref{lem2.2}, we obtain that $u_{k}$ is
bounded near $0$, which contradicts the assumption that $\sup_{k}c_{k}=+\infty$. Thus $\underset{k\rightarrow\infty}{\lim}\left\Vert u_{k}\right\Vert _{{n}}^{{n}}=0$ and the proof of Lemma \ref{lem4.5} has been finished.
\end{proof}

\begin{remark}
\label{remark}\ From Lemma \ref{lem4.5}, one can obtain the following results.
\[
\underset{k}{\lim}\alpha_{k}=\lambda_{n},\ \underset{k}{\lim}\mu_{k}=1,
\]%
\[
\underset{k\rightarrow\infty}{\lim\sup}\int_{%
\mathbb{R}
^{n}}\left(  F^{n}(\nabla u_{k}^{A})+\left\vert u_{k}%
^{A}\right\vert ^{n}\right)  dx=\frac{1}{A},\,
\]%
\[
\psi\left(  x\right)  =-n\log\left(  1+\kappa_{n}
^{\frac{1}{n-1}}F^{0}(x)  ^{\frac{n}{n-1}}\right)  ,
\]
and
\begin{align}
&  \underset{R\rightarrow\infty}{\lim}\underset{k\rightarrow\infty}{\lim}
\frac{1}{\lambda_{k}}\int_{\mathcal{W}_{Rr_{k}}}u_{k}^{\frac{n}{n-1}}\exp\left(
\alpha_{k}u_{k}^{\frac{n}{n-1}}\right)  dx=\underset{R\rightarrow\infty}{\lim
}\underset{k\rightarrow\infty}{\lim}\frac{1}{\mu_{k}}\int_{\mathcal{W}_{R}}
e^{\psi\left(  x\right)  }dx\nonumber\\
=&\underset{k\rightarrow\infty}{\lim}\frac{1}{ \mu_{k}\left(1+\alpha\underset
{k}{\lim}\left\Vert u_{k}\right\Vert _{{n}}^{{n}}\right)
}=1.\label{add9}
\end{align}

\end{remark}

\begin{corollary}
\label{tent to 0}We have $\underset{k\rightarrow\infty}{\lim}\int_{
\mathbb{R}
^{n}\backslash \mathcal{W}_{\delta}}\left(  F^{n}(\nabla u_{k})+\left\vert u_{k}\right\vert ^{n}\right)  dx=0$ for any $\delta>0$, and
then $\underset{k\rightarrow\infty}{\lim}u_{k}\equiv0$.
\end{corollary}

\begin{lemma}\label{lem4.6}
\ There holds
\begin{equation}
\underset{k\rightarrow\infty}{\lim}\int_{
\mathbb{R}
^{n}}\Phi\left(  \alpha_{k}u_{k}^{\frac{n}{n-1}}\right)  dx\leq\underset
{R\rightarrow\infty}{\lim}\underset{k\rightarrow\infty}{\lim}\int_{\mathcal{W}_{Rr_{k}}
}\left(  \exp\left(  \alpha_{k}u_{k}^{\frac{n}{n-1}}\right)  -1\right)
dx=\underset{k\rightarrow\infty}{\lim\sup}\frac{\lambda_{k}}{c_{k}^{\frac
{n}{n-1}}}. \label{7}
\end{equation}
Moreover,
\begin{equation}
\frac{\lambda_{k}}{c_{k}}\rightarrow\infty\text{ and }\underset{k\rightarrow\infty}{\sup}
\frac{c_{k}^{\frac{n}{n-1}}}{\lambda_{k}}\leq\infty. \label{7.1}%
\end{equation}
\end{lemma}

\begin{proof}
\bigskip For any $A>1$, it follows from the expression of $\lambda_{k}$ that
\begin{align*}
\int_{
\mathbb{R}
^{n}}\Phi\left(  \alpha_{k}u_{k}^{\frac{n}{n-1}}\right)  dx &  \leq\int
_{u_{k}<\frac{c_{k}}{A}}\Phi\left(  \alpha_{k}u_{k}^{\frac{n}{n-1}}\right)
dx+\int_{u_{k}\geq\frac{c_{k}}{A}}\Phi^{\prime}\left(  \alpha_{k}u_{k}%
^{\frac{n}{n-1}}\right)  dx\\
&  \leq\int_{
\mathbb{R}
^{n}}\Phi\left(  \alpha_{k}\left(  u_{k}^{A}\right)  ^{\frac{n}{n-1}}\right)
dx+\int_{u_{k}\geq\frac{c_{k}}{A}}\Phi^{\prime}\left(  \alpha_{k}u_{k}%
^{\frac{n}{n-1}}\right)  dx\\
&  \leq\int_{
\mathbb{R}
^{n}}\Phi\left(  \alpha_{k}\left(  u_{k}^{A}\right)  ^{\frac{n}{n-1}}\right)
dx+\left(  \frac{A}{c_{k}}\right)  ^{\frac{n}{n-1}}\lambda_{k}\int_{u_{k}%
\geq\frac{c_{k}}{A}}\frac{u_{k}^{\frac{n}{n-1}}}{\lambda_{k}}\Phi^{\prime
}\left(  \alpha_{k}u_{k}^{\frac{n}{n-1}}\right)  dx.
\end{align*}
By Remark \ref{remark} and Theorem 1.1 in \cite{ZZ2}, we obtain $\Phi\left(
\alpha_{k}\left(  u_{k}^{A}\right)  ^{\frac{n}{n-1}}\right)  $ is bounded in
$L^{r}$ for some $r>1$. Since $u_{k}^{A}\rightarrow0$ a.e. in $
\mathbb{R}
^{n}$ as $k\rightarrow\infty$, it follows
\[
\int_{
\mathbb{R}
^{n}}\Phi\left(  \alpha_{k}\left(  u_{k}^{A}\right)  ^{\frac{n}{n-1}}\right)
dx\rightarrow\int_{
\mathbb{R}
^{n}}\Phi\left(  0\right)  dx=0\text{, as }k\rightarrow\infty\text{.}
\]
By (\ref{add9}), then we have
\begin{align*}
\underset{k\rightarrow\infty}{\lim}\int_{
\mathbb{R}
^{n}}\Phi\left(  \alpha_{k}u_{k}^{\frac{n}{n-1}}\right)  dx &  \leq\left(
\frac{A}{c_{k}}\right)  ^{\frac{n}{n-1}}\lambda_{k}\int_{u_{k}\geq\frac{c_{k}%
}{A}}\frac{u_{k}^{\frac{n}{n-1}}}{\lambda_{k}}\Phi^{\prime}\left(  \alpha
_{k}u_{k}^{\frac{n}{n-1}}\right)  dx+o\left(  1\right)  \\
&  =\underset{k\rightarrow\infty}{\lim}A^{\frac{n}{n-1}}\frac{\lambda_{k}%
}{c_{k}^{\frac{n}{n-1}}}+o\left(  1\right).
\end{align*}
Letting $A\rightarrow1$ and $k\rightarrow\infty$, we get (\ref{7}).

\medskip

If $\frac{\lambda_{k}}{c_{k}}$ is bounded or $\underset{k\rightarrow\infty}{\sup}\frac
{c_{k}^{\frac{n}{n-1}}}{\lambda_{k}}=\infty$, it follows from (\ref{7}) that
\[
\underset{k\rightarrow\infty}{\lim}\int_{%
\mathbb{R}
^{n}}\Phi\left(  \alpha_{k}u_{k}^{\frac{n}{n-1}}\right)  dx=0,
\]
which is impossible and the proof of Lemma \ref{lem4.6} is completed.
\end{proof}

\begin{lemma}\label{lem4.7}
\bigskip\label{dirac}For any $\varphi\in C_{0}^{\infty}\left(
\mathbb{R}
^{n}\right) $, there holds
\end{lemma}

\[
\bigskip\int_{
\mathbb{R}
^{n}}\varphi\mu_{k}\lambda_{k}^{-1}c_{k}u_{k}^{\frac{1}{n-1}}\Phi^{\prime
}\left(  \alpha_{k}u_{k}^{\frac{n}{n-1}}\right)  dx=\varphi\left(  0\right)
.
\]

\begin{proof}
We adopt the method for the proof of Lemma 3.6 in \cite{liruf}.  Split the integral as follows
\begin{align*}
&\int_{
\mathbb{R}
^{n}}\varphi\mu_{k}\lambda_{k}^{-1}c_{k}u_{k}^{\frac{1}{n-1}}\Phi^{\prime
}\left(  \alpha_{k}\left(  u_{k}\right)  ^{\frac{n}{n-1}}\right)  dx\\
\leq&\left(  \int_{\left\{  u_{k}\geq\frac{c_{k}}{A}\right\}  \backslash
\mathcal{W}_{Rr_{k}}}+\int_{\mathcal{W}_{Rr_{k}}}+\int_{\left\{  u_{k}<\frac{c_{k}}{A}\right\}
}\right) \varphi\mu_{k}\lambda_{k}^{-1}c_{k}u_{k}^{\frac{1}{n-1}}\Phi^{\prime
}\left(  \alpha_{k}\left(  u_{k}\right)  ^{\frac{n}{n-1}}\right) dx\\
=:&I_{1}+I_{2}+I_{3}.
\end{align*}
Then
\begin{align*}
I_{1} &  \leq A\left\Vert \varphi\right\Vert _{L^{\infty}}\int_{\left\{
u_{k}\geq\frac{c_{k}}{A}\right\}  \backslash \mathcal{W}_{Rr_{k}}}\mu_{k}\lambda
_{k}^{-1}c_{k}u_{k}^{\frac{1}{n-1}}\Phi^{\prime}\left(  \alpha_{k}\left(
u_{k}\right)  ^{\frac{n}{n-1}}\right)  dx\\
&  \leq A\left\Vert \varphi\right\Vert _{L^{\infty}}\left(  \int_{
\mathbb{R}
^{n}}-\int_{\mathcal{W}_{Rr_{k}}}\right)  \mu_{k}\lambda_{k}^{-1}u_{k}^{\frac{n}{n-1}%
}\Phi^{\prime}\left(  \alpha_{k}\left(  u_{k}\right)  ^{\frac{n}{n-1}}\right)
dx\\
&  \leq A\left\Vert \varphi\right\Vert _{L^{\infty}}\left(  1-\int_{\mathcal{W}_{R}}%
\exp\left(  \alpha_{k}m_{k}^{\frac{n}{n-1}}-\alpha_{k}c_{k}^{\frac{n}{n-1}%
}\right)  \right)  \\
&  =A\left\Vert \varphi\right\Vert _{L^{\infty}}\left(  1-\int_{\mathcal{W}_{R}}%
\exp\left(  \psi_{k}\left(  x\right)  +o\left(  1\right)  \right)  \right).
\end{align*}
For $I_{2}$, we have
\begin{align*}
I_{2} &  =\int_{\mathcal{W}_{Rr_{k}}}\varphi\mu_{k}\lambda_{k}^{-1}c_{k}u_{k}^{\frac
{1}{n-1}}\Phi^{\prime}\left(  \alpha_{k}u_{k}^{\frac{n}{n-1}}\right)  dx\\
&  =\int_{\mathcal{W}_{R}}\varphi\left(  r_{k}x\right)  \left(  \frac{m_{k}}{c_{k}%
}\right)  ^{\frac{1}{n-1}}\exp\left(  \alpha_{k}m_{k}^{\frac{n}{n-1}}%
-\alpha_{k}c_{k}^{\frac{n}{n-1}}\right)  dx+o(1)\\
&  =\varphi\left(  0\right)  \int_{\mathcal{W}_{R}}\exp\left(  \psi_{k}\left(  x\right)
+o\left(  1\right)  \right)  dx+o\left(  1\right)  =\varphi\left(  0\right)
+o\left(  1\right)  \rightarrow\varphi\left(  0\right)  \text{, as
}k\rightarrow\infty\text{.}
\end{align*}
By Lemma \ref{lem4.6} and H\"{o}lder's inequality, it follows
\begin{align*}
I_{3} &  =\int_{\left\{  u_{k}<\frac{c_{k}}{A}\right\}  }\varphi\mu_{k}%
\lambda_{k}^{-1}c_{k}u_{k}^{\frac{1}{n-1}}\Phi^{\prime}\left(  \alpha
_{k}\left(  u_{k}\right)  ^{\frac{n}{n-1}}\right)  dx\\
&  =\int_{%
\mathbb{R}
^{n}}\varphi\mu_{k}\lambda_{k}^{-1}c_{k}\left(  u_{k}^{A}\right)  ^{\frac
{1}{n-1}}\Phi^{\prime}\left(  \alpha_{k}\left(  u_{k}^{A}\right)  ^{\frac
{n}{n-1}}\right)  dx\\
&  \leq c_{k}\left\Vert \varphi\right\Vert _{L^{\infty}}\lambda_{k}%
^{-1}\left(  \int_{%
\mathbb{R}
^{n}}\left(  u_{k}^{A}\right)  ^{\frac{q}{n-1}}dx\right)  ^{\frac{1}{q}%
}\left(  \int_{
\mathbb{R}
^{n}}\Phi^{\prime}\left(  q^{\prime}\alpha_{k}\left(  u_{k}^{A}\right)
^{\frac{n}{n-1}}\right)  dx\right)  ^{\frac{1}{q^{\prime}}}\rightarrow0,\text{
as }k\rightarrow\infty,
\end{align*}
for any $q^{\prime}<A^{\frac{1}{n-1}}$ such that $q=\frac{q^{\prime}%
}{q^{\prime}-1}$ large enough. Letting $R\rightarrow+\infty$, by Remark
\ref{remark}, then Lemma \ref{lem4.7} is proved.
\end{proof}

\begin{lemma}\label{tend to G 1}
On any $\Omega\subset\subset
\mathbb{R}
^{n}\backslash\{0\}$, we have $c_{k}^{\frac{1}{n-1}}u_{k}\rightarrow G_{\alpha} $ in $C^{1}\left(  \Omega\right)
$, where $G_{\alpha}\in
C^{1,\alpha}_{loc}( \mathbb{R}^{n} \backslash\{0\})$ is a Green function satisfying the following equation
\begin{equation}
-Q_{n}(G_{\alpha})=\delta_{0}+\left(  \alpha-1\right)  G_{\alpha}^{n-1}.
\label{121}
\end{equation}
\end{lemma}

\begin{proof}
The idea of the proof is from Struwe \cite{S1} (also
see \cite{liruf}).
Denote $U_{k}=c_{k}^{\frac{1}{n-1}}u_{k}$, by (\ref{equation}), $U_{k}$
satisfy the following equation
\begin{equation}
-Q_{n}U_{k}=\mu_{k}c_{k}\lambda_{k}^{-1}u_{k}^{\frac{1}{n-1}}
\Phi^{\prime}\left\{  \alpha_{k}u_{k}^{\frac{n}{n-1}}\right\}  +\left(
\gamma_{k}-1\right)  U_{k}^{n-1}. \label{add 8}
\end{equation}

For $t\geq1$, let $\ U_{k}^{t}=\min\left\{  U_{k},t\right\}$ and
$\Omega_{t}^{k}=\left\{  0\leq U_{k}\leq t\right\}  $. Testing (\ref{add 8})
with $U_{k}^{t}$, it follows
\[
\int_{
\mathbb{R}
^{n}}-U_{k}^{t}Q_{n}(U_{k})dx+\left(  1-\gamma_{k}\right)  \int_{
\mathbb{R}
^{n}}U_{k}^{t}U_{k}^{n-1}dx\leq\int_{
\mathbb{R}
^{n}}U_{k}^{t}\mu_{k}c_{k}\lambda_{k}^{-1}u_{k}^{\frac{1}{n-1}}\Phi^{\prime
}\left\{  \alpha_{k}u_{k}^{\frac{n}{n-1}}\right\}  dx.
\]

By the fact that $\gamma_{k}\rightarrow\alpha<1$ as $k\rightarrow\infty$, we obtian
\begin{align*}
\int_{\Omega_{t}^{k}}F^{n}( \nabla U_{k}^{t})
dx+\int_{\Omega_{t}^{k}}\left\vert U_{k}^{t}\right\vert ^{n}dx &  \leq\int_{
\mathbb{R}
^{n}}\left(  -U_{k}^{t}Q_{n}(U_{k})dx+U_{k}^{t}U_{k}^{n-1}\right)  dx\\
&  \leq C\int_{\mathbb{R}
^{n}}U_{k}^{t}\mu_{k}c_{k}\lambda_{k}^{-1}u_{k}^{\frac{1}{n-1}}\Phi^{\prime
}\left\{  \alpha_{k}u_{k}^{\frac{n}{n-1}}\right\}  dx\leq Ct.
\end{align*}

Let $\eta$ be a radially symmetric cut-off function which is $1$ on $\mathcal{W}_{R/2}$
and $0$ on $\mathcal{W}_{R}^{c}$, and satisfy $F(\nabla\eta)\leq\frac{C}{R}$.
Then when $R$ large enough, we obtian
\[
\int_{\mathcal{W}_{R}}F^{n} (\nabla(\eta U_{k}^{t}))dx \leq\int
_{\mathcal{W}_{R}}F^{n} (\nabla\eta)\left\vert U_{k}^{t}\right\vert
^{n}dx+\int_{\mathcal{W}_{R}}\eta^{n}F^{n} (\nabla U_{k}^{t}) dx\leq
C_{1}\left(  R\right)  t+C_{2}\left(  R\right).
\]
Taking $t$ large enough such that $C_{1}(R)  t>C_{2}(R)$, then
\[
\int_{\mathcal{W}_{R}}F^{n}(\nabla (\eta U_{k}^{t}))dx\leq2C_{1}\left(
R\right) t.
\]

Let $|\mathcal{W}_{\rho}|=|\{x\in \mathcal{W}_{R}: U_{k}>t\}|$, then
\begin{equation}
\inf_{\{\psi\in W_{0}^{1,n}(\mathcal{W}_{R}),\ \psi|_{\mathcal{W}_{\rho}}=t\}}\int_{\mathcal{W}_{R}}F^{n}(\nabla \psi)dx\leq \int_{\mathcal{W}_{R}}F^{n}(\nabla (\eta U_{k}^{t}))dx\leq2C_{1}\left(
R\right) t.\label{113}
\end{equation}
The above infimum can be attained (see \cite{ZZ1}) by

\[
\psi_{1}(x)=\left\{
\begin{array}
[c]{c}
t\log\frac{R}{F_{0}(x)}/\log\frac{R}{\rho},\ \ \text{in}
\mathcal{W}_{R}\backslash\mathcal{W}_{\rho},\\
t,\ \ \ \ \text{in }\mathcal{W}_{\rho}.\\
\end{array}
\right.
\]
By computing $||F(\nabla \psi_{1})||^{n}_{L^{n}(\mathcal{W}_{R})}$, then it follows from (\ref{113}) that $\rho \leq CR^{-C_{3}t}$.
Thus
\begin{equation*}
 |\{x\in\mathcal{W}_{R}:U_{k}\geq t \}|=|\mathcal{W}_{\rho}|\leq\kappa_{n}R^{n}e^{-nC_{3}t}.
\end{equation*}

For any $0<\delta<nC_{3}$, we have
\begin{align*}
\int_{\mathcal{W}_{R}}e^{\delta U_{k}}dx\leq &e^{\delta}|\mathcal{W}_{R}|+\sum_{m=1}^{\infty}e^{(m+1)\delta}|\{x\in \mathcal{W}_{R}: m\leq U_{k}\leq m+1\}|\\
\leq& e^{\delta}|\mathcal{W}_{R}|+ \kappa_{n}R^{n}e^{\delta} \sum_{m=1}^{\infty}e^{-(nC_{3}-\delta)m}\leq C.
\end{align*}

Testing (\ref{add 8}) with $\log\frac{1+2U_{k}}{1+U_{k}}$, we get
\begin{align*}
&\int_{\mathcal{W}_{R}}\frac{F^{n}(\nabla U_{k})}{(1+2U_{k})(1+U_{k})}dx\\
\leq& \log2 \int_{\mathcal{W}_{R}}\mu_{k}c_{k}\lambda_{k}^{-1}u_{k}^{\frac{1}{n-1}}\Phi^{'}(\beta_{k}u_{k}^{\frac{n}{n-1}})dx
+\int_{\mathcal{W}_{R}}(\gamma_{k}-1)U_{k}^{n-1}\log\frac{1+2U_{k}}{1+U_{k}}dx\leq C.
\end{align*}

For any $1<q<n$, it follows by the Young inequality that

\begin{align*}
\int_{\mathcal{W}_{R}}F^{q}(\nabla U_{k})dx\leq &\int_{\mathcal{W}_{R}}\frac{F^{n}(\nabla U_{k})}{(1+2U_{k})(1+U_{k})}dx+\int_{\mathcal{W}_{R}}\left\{(1+2U_{k})(1+U_{k})\right\}^{\frac{q}{n-q}}dx\\
\leq& C(1+\int_{\mathcal{W}_{R}}e^{\delta U_{k}}dx)\leq C.
\end{align*}

Then we can obtain that $\left\Vert F(\nabla
U_{k})\right\Vert _{L^{q}\left( \mathcal{ W}_{R}\right)  }\leq C$ for any $1<q<n$, thus $\left\Vert
U_{k}\right\Vert _{L^{p}\left(  \mathcal{ W}_{R}\right) }\leq+\infty$ for any
$0<p<+\infty$. By Corollary \ref{tent to 0}, we know that $\exp\left\{
\alpha_{k}u_{k}^{\frac{n}{n-1}}\right\}  $ is bounded in
$L^{r}\,\left(  \Omega\backslash\left\{ \mathcal{ W}_{\delta}\right\}  \right)
$ for any $r>0$ and $\delta>0$. Applying Theorem 2 in \cite{S4}
and Theorem 1 in \cite{Tolksdorf}, we have $\left\Vert U_{k}\right\Vert
_{C^{1,\alpha}\left( \mathcal{W}_{R}\right)  }\leq C$, then
$c_{k}^{\frac{1}{n-1}}u_{k}\rightarrow  G_{\alpha }$ in
$C^{1}\left( \mathcal{ W}_{R}\right)  $. So we complete the proof of Lemma \ref{tend to G 1}.
\end{proof}

   Similar as Lemma 3.8 in \cite{liruf} or Lemma 4.9 in \cite{LuZhu}, we can obtain the following
asymptotic representation of $G_{\alpha}$.
\begin{lemma}\label{tend to G}
$G_{\alpha}\in C_{loc}^{1,\beta}\left(
\mathbb{R}
^{n}\backslash\left\{  0\right\}  \right)  $ for some $\beta>0$, and near $0$ we can write

\begin{equation}
 G_{\alpha}=-\frac{n}{\lambda_{n}}\log r+A+O\left(  r^{n}\log
^{n}r\right)  ,\label{13}
\end{equation}
where $A$ is a constant and $r=F^{0}(x)$.
Moreover, for any $\delta>0$, it holds
\begin{equation}
\underset{k\rightarrow\infty}{\lim}\left(  \int_{\mathbb{R}
^{n}\backslash \mathcal{W}_{\delta}}F ^{n}( \nabla U_{k}) dx+\left(
1-\alpha\right)  \int_{\mathbb{R}
^{n}\backslash \mathcal{W}_{\delta}}U_{k}^{n}dx\right)
=G_{\alpha}(\delta)\left(1+(\alpha-1)\int_{\mathcal{W}_\delta}G_{\alpha}^{n-1}dx\right).
\label{add 9}
\end{equation}

\end{lemma}

\begin{proof}
The proof of (\ref{13}) is similar as Lemma 4.7 in \cite{ZZ2}, here we omit the details. Now we
give the proof of (\ref{add 9}).
By Corollary \ref{tent to 0}, we have
\begin{equation}
\int_{
\mathbb{R}
^{n}\backslash \mathcal{W}_{\delta}}u_{k}^{\frac{n}{n-1}}\Phi^{\prime}\left\{
\alpha_{k}u_{k}^{\frac{n}{n-1}}\right\}  dx\leq C\int_{
\mathbb{R}
^{n}\backslash \mathcal{W}_{\delta}}u_{k}^{n}dx\rightarrow0.\label{14}%
\end{equation}
Testing (\ref{add 8}) with $U_{k}$,
\begin{align*}
&\int_{\mathbb{R}
^{n}\backslash \mathcal{W}_{\delta}}F^{n}( \nabla U_{k})%
dx+\int_{\partial \mathcal{W}_{\delta}}F ^{n-1}( \nabla U_{k})%
U_{k}\frac{\partial U_{k}}{\partial n}dx\\
=&\int_{\mathbb{R}
^{n}\backslash \mathcal{W}_{\delta}}\mu_{k}c_{k}^{\frac{n}{n-1}}\lambda_{k}^{-1}%
u_{k}^{\frac{n}{n-1}}\Phi^{\prime}\left\{  \alpha_{k}u_{k}^{\frac{n}{n-1}%
}\right\}  dx+\int_{\mathbb{R}
^{n}\backslash \mathcal{W}_{\delta}}\left(  \gamma_{k}-1\right)  U_{k}^{n}dx.
\end{align*}
By (\ref{14}), (\ref{7.1}), it follows that
\begin{align*}
\underset{k\rightarrow\infty}{\lim}\int_{
\mathbb{R}
^{n}\backslash \mathcal{W}_{\delta}}F^{n} (\nabla U_{k})dx &
=-\underset{k\rightarrow\infty}{\lim}\int_{\partial \mathcal{W}_{\delta}}F^{n-1}
(\nabla U_{k})U_{k}\frac{\partial U_{k}}{\partial n}dx+\left(
\alpha-1\right)  \underset{k\rightarrow\infty}{\lim}\int_{
\mathbb{R}^{n}\backslash \mathcal{W}_{\delta}}U_{k}^{n}dx\\
&  =-G_{\alpha}\left(  \delta\right)  \int_{\partial \mathcal{W}_{\delta}}F ^{n-1}(
\nabla G_{\alpha})\frac{\partial G_{\alpha}}{\partial
n}dx+\left(  \alpha-1\right)  \underset{k\rightarrow\infty}{\lim}\int_{
\mathbb{R}
^{n}\backslash \mathcal{W}_{\delta}}U_{k}^{n}dx\\
&  =G_{\alpha}(\delta)\left(1+(\alpha-1)\int_{\mathcal{W}_\delta}G_{\alpha}^{n-1}dx\right)+\left(  \alpha-1\right)  \underset
{k\rightarrow\infty}{\lim}\int_{
\mathbb{R}
^{n}\backslash \mathcal{W}_{\delta}}U_{k}^{n}dx.
\end{align*}
Thus
\[
\underset{k\rightarrow\infty}{\lim}\left(  \int_{
\mathbb{R}^{n}\backslash \mathcal{W}_{\delta}}F^{n}(\nabla U_{k})dx+\left(
1-\alpha\right)  \int_{
\mathbb{R}^{n}\backslash \mathcal{W}_{\delta}}U_{k}^{n}dx\right)
=G_{\alpha}(\delta)\left(1+(\alpha-1)\int_{\mathcal{W}_\delta}G_{\alpha}^{n-1}dx\right).
\]
The proof of Lemma \ref{tend to G} is completed.
\end{proof}

\begin{proof}
[Proof of the first part of Theorem \ref{moser-trudinger}]By
(\ref{2}), there exist some $L>0$ such that $u_{k}\left(  L\right)  <1$,
then
\[
\int_{\mathbb{R}
^{n}\backslash \mathcal{W}_{L}}\exp\left\{  \beta_{k}\left\vert u_{k}\right\vert
^{\frac{n}{n-1}}\left(  1+\alpha\left\Vert u_{k}\right\Vert _{{n}}^{{n}%
}\right)  ^{\frac{1}{n-1}}\right\}  dx\leq C\int_{
\mathbb{R}^{n}\backslash \mathcal{W}_{L}}\left\vert u_{k}\right\vert ^{n}dx\leq C.
\]
Since $\left(  u_{k}-u_{k}\left(
L\right)  \right)  ^{+}\in W_{0}^{1,n}\left(  B_{L}\right)  $, then
\begin{align*}
u_{k}^{\frac{n}{n-1}} &  =\left(  \left(  u_{k}-u_{k}\left(  L\right)
\right)  ^{+}+u_{k}\left(  L\right)  \right)  ^{\frac{n}{n-1}}\\
&  \leq\left(  \left(  u_{k}-u_{k}\left(  L\right)  \right)  ^{+}\right)
^{\frac{n}{n-1}}+C\left(  \left(  u_{k}-u_{k}\left(  L\right)  \right)
^{+}\right)  ^{\frac{1}{n-1}}u_{k}\left(  L\right)  +u_{k}\left(  L\right)
^{\frac{n}{n-1}}.
\end{align*}
By Lemma \ref{tend to G 1}, we obtain $c_{k}^{\frac{1}{n-1}}u_{k}\rightharpoonup
G_{\alpha}$, then $u_{k}\left(  L\right)  =\frac{G_{\alpha}\left(  L\right)
}{c_{k}^{\frac{1}{n-1}}}$. Thus
\begin{align*}
u_{k}^{\frac{n}{n-1}} &  \leq\left(  \left(  u_{k}-u_{k}\left(  L\right)
\right)  ^{+}\right)  ^{\frac{n}{n-1}}+C\left(  \frac{\left(  u_{k}
-u_{k}\left(  L\right)  \right)  ^{+}}{c_{k}}\right)  ^{\frac{1}{n-1}}
+u_{k}\left(  L\right)  ^{\frac{n}{n-1}}\\
&  \leq\left(  \left(  u_{k}-u_{k}\left(  L\right)  \right)  ^{+}\right)
^{\frac{n}{n-1}}+C.
\end{align*}
Therefore
\begin{align*}
&  \int_{\mathcal{W}_{L}}\exp\left\{  \beta_{k}\left\vert u_{k}\right\vert ^{\frac
{n}{n-1}}\left(  1+\alpha\left\Vert u_{k}\right\Vert _{{n}}^{{n}}\right)
^{\frac{1}{n-1}}\right\}  dx\\
&  \leq C\int_{\mathcal{W}_{L}}\exp\left\{  \beta_{k}\left(  \left(  u_{k}-u_{k}\left(
L\right)  \right)  ^{+}\right)  ^{\frac{n}{n-1}}\left(  1+\alpha\left\Vert
u_{k}\right\Vert _{{n}}^{{n}}\right)  ^{\frac{1}{n-1}}\right\}  dx\\
&  \leq C\int_{\mathcal{W}_{L}}\exp\left\{  \beta_{k}\left(  \left(  u_{k}-u_{k}\left(
L\right)  \right)  ^{+}\right)  ^{\frac{n}{n-1}}\left(  \left(  1+\alpha
\left\Vert u_{k}\right\Vert _{{n}}^{{n}}\right)  ^{\frac{1}{n-1}}-1\right)
\right\}  \exp\left(  \beta_{k}\left(  \left(  u_{k}-u_{k}\left(  L\right)
\right)  ^{+}\right)  ^{\frac{n}{n-1}}\right)  dx\\
&  \leq C\exp\left\{  \beta_{k}c_{k}^{\frac{n}{n-1}}\left(  \left(
1+\alpha\left\Vert u_{k}\right\Vert _{{n}}^{{n}}\right)  ^{\frac{1}{n-1}%
}-1\right)  \right\}  \int_{\mathcal{W}_{L}}\exp\left(  \beta_{k}\left(  \left(
u_{k}-u_{k}\left(  L\right)  \right)  ^{+}\right)  ^{\frac{n}{n-1}}\right)
dx.
\end{align*}
From Lemma \ref{tend to G 1} and Lemma \ref{tend to G}, we obtain that $\Vert
c_{k}^{\frac{1}{n-1}}u_{k}\Vert _{{n}}$ is bounded. Applying the anisotropic Moser-Trudinger inequality (see \cite{WX}), by the fact
that $\left\Vert u_{k}\right\Vert _{{n}}^{{n}}\rightarrow0$, we have
\begin{align*}
&  \int_{\mathcal{W}_{L}}\exp\left\{  \beta_{k}\left\vert u_{k}\right\vert ^{\frac
{n}{n-1}}\left(  1+\alpha\left\Vert u_{k}\right\Vert _{{n}}^{{n}}\right)
^{\frac{1}{n-1}}\right\}  dx\\
\leq& C\exp\left\{  \frac{\alpha\beta_{k}c_{k}^{\frac{n}{n-1}}}%
{n-1}\left\Vert u_{k}\right\Vert _{{n}}^{{n}}\right\}  \int_{\mathcal{W}_{L}}\exp\left(
\beta_{k}\left(  \left(  u_{k}-u_{k}\left(  L\right)  \right)  ^{+}\right)
^{\frac{n}{n-1}}\right)  dx\\
=&C\exp\left\{  \frac{\alpha\beta_{k}}{n-1}\Vert c_{k}^{\frac{1}{n-1}%
}u_{k}\Vert _{{n}}^{{n}}\right\}  \int_{\mathcal{W}_{L}}\exp\left(  \beta
_{k}\left(  \left(  u_{k}-u_{k}\left(  L\right)  \right)  ^{+}\right)
^{\frac{n}{n-1}}\right)  dx\\
\leq& C.
\end{align*}
Thus we complete the proof of the first part of Theorem \ref{moser-trudinger} in the case of $\sup_{k}c_{k}=+\infty$.
\end{proof}

\section{\bigskip Proof of Theorem 1.2}
In this section, we prove the proof of Theorem 1.2 in this paper by considering the two cases. When $sup_{k}c_{k}<+\infty$, the proof is based on the concentration-compactness lemma. When $sup_{k}c_{k}=+\infty$, we prove the result by contradiction.
We first establish the upper bound for critical functional when $\sup_{k}c_{k}=+\infty$, and then construct an explicit test function, which provides a lower bound for the
supremum of our Moser-Trudinger inequality. Because this lower bound equals to the upper bound, one can obtain the contradiction.

\subsection{Proof in the case of $sup_{k}c_{k}<+\infty$}

\begin{theorem}\label{theorem5.1}
\label{attain lemma2}\bigskip If $\ \underset{k}{\sup}c_{k}<+\infty$,
then Theorem \ref{attain} holds.
\end{theorem}

\begin{proof}
By Lemma \ref{lamna} and applying
the elliptic estimate in \cite{Tolksdorf} to equation (\ref{equation}), we can obtian
$u_{k}\rightarrow u$ in
$C_{loc}^{1}\left(\mathbb{R}^{n}\right)$. Next we will prove $u\neq0$. We prove this result by contradiction.

Assume $u=0$, we claim that $\left\{  u_{k}\right\}  $ is not a
Sobolev-normalized concentrating sequence. If not, i.e. $\left\{  u_{k}\right\}  $ is a
Sobolev-normalized concentrating sequence, by iii) of
Definition \ref{SNC} and the fact that $\left\vert u_{k}\right\vert
$ is bounded, we have for any $\delta>0$,
$$
\int_{
\mathbb{R}
^{n}}u_{k}^{n}dx  \leq\int_{\mathcal{W}_{\delta}}u_{k}^{n}dx+\int_{
\mathbb{R}
^{n}\backslash\mathcal{W}_{\delta}}u_{k}^{n}dx
\leq C\delta^{n}+o_{k}\left(  1\right)  .
$$

Letting $\delta\rightarrow0$, it follows $\int_{
\mathbb{R}
^{n}}u_{k}^{n}dx\rightarrow0$ as $k\rightarrow\infty$. When $L$ is large enough, for any $\varepsilon>0$, it follows by (\ref{33}) that%
\begin{align*}
&  S+o_{k}\left(  1\right)  =\int_{
\mathbb{R}
^{n}}\Phi\left(  \alpha_{k}u_{k}^{\frac{n}{n-1}}\right)  dx\\
=&\int_{
\mathbb{R}
^{n}}\frac{\alpha_{k}^{n-1}u_{k}^{n}}{\left(  n-1\right)  !}dx+\int_{\mathcal{W}_{L}
}\left(  \Phi\left(
\alpha_{k}\cdot
u^{\frac{n}{n-1}}\right)
-\int_{\mathcal{W}_{L}}\frac{\alpha_{k}^{n-1}\cdot u^{n}}{\left(  n-1\right)  !}\right)
dx+O\left( \varepsilon^{\frac{n^{2}}{n-1}-n}\right).
\end{align*}
Thus
\[
S\leq\int_{
\mathbb{R}
^{n}}\frac{\alpha_{k}^{n-1}u_{k}^{n}}{\left(  n-1\right)
!}dx\rightarrow0,
\]
which is impossible. Therefore the claim is proved, i.e. when $u=0$, we have $\left\{  u_{k}\right\}  $ is not a
Sobolev-normalized concentrating sequence. By Lemma \ref{lions}, it follows that $\int_{
\mathbb{R}
^{n}}\Phi\left(  \alpha_{k}u_{k}^{\frac{n}{n-1}}\right)  dx\rightarrow\int_{
\mathbb{R}
^{n}}\Phi\left(  \alpha_{n}u^{\frac{n}{n-1}}\right)  dx=0$, which is
still  impossible. Thus $u\neq0$.

 Next we will prove that $\int_{\mathbb{R}
^{n}}u_{k}^{n}\rightarrow\int_{\mathbb{R}
^{n}}u^{n}$. By (\ref{33}), we get
\begin{align}
S=&\underset{k\rightarrow\infty}{\lim}\int_{
\mathbb{R}
^{n}}\Phi\left(  \alpha_{k}u_{k}^{\frac{n}{n-1}}\right)  dx\nonumber\\
=& \int_{
\mathbb{R}
^{n}}\left(  \Phi\left(  \underset{k\rightarrow\infty}{\lim}\alpha_{k}
u^{\frac{n}{n-1}}\right)  \right)  dx+\underset{k\rightarrow\infty}{\lim}
\int_{
\mathbb{R}
^{n}}\frac{\underset{k\rightarrow\infty}{\lim}\alpha_{k}^{n-1}\left(
u_{k}^{n}-u^{n}\right)  }{\left(  n-1\right)  !}dx.\label{4}
\end{align}
Denote
\[
\tau^{n}=\lim_{k\rightarrow+\infty}\frac{\int_{
\mathbb{R}
^{n}}u_{k}^{n}}{\int_{
\mathbb{R}
^{n}}u^{n}}.
\]
By the Levi Lemma, it is easy to see that $\tau\geq1$. \ Set $\tilde{u}=u\left(  \frac
{x}{\tau}\right)  $, then it follows
\[
\int_{\mathbb{R}^{n}}F^{n}(\nabla\tilde{u})dx=\int_{
\mathbb{R}^{n}} F^{n}(\nabla u)dx\leq\int_{
\mathbb{R}^{n}}F^{n}( \nabla u_{k})dx
\]
and
\[
\int_{\mathbb{R}
^{n}}\left\vert \tilde{u}\right\vert ^{n}dx=\tau^{n}\int_{%
\mathbb{R}
^{n}}\left\vert u\right\vert ^{n}dx\leq\int_{%
\mathbb{R}
^{n}}\left\vert u_{k}\right\vert ^{n}dx.
\]
Thus
\[
\Vert\tilde{u}\Vert_{F}=\int_{
\mathbb{R}
^{n}}\left( F^{n}(\nabla\tilde{u})+\left\vert \tilde
{u}\right\vert ^{n}\right)  dx\leq1.
\]

By (\ref{4}), it follows that

\begin{align*}
S &  \geq\int_{
\mathbb{R}
^{n}}\Phi\left( \lambda_{n}\tilde{u}^{\frac{n}{n-1}} \left(  1+\alpha\left\Vert \tilde{u}\right\Vert
_{n}^{n}\right)^{\frac{1}{n-1}}\right)  dx\\
&  =\tau^{n}\int_{
\mathbb{R}
^{n}}\Phi\left(  \lambda_{n}u^{\frac{n}{n-1}}\left(  1+\alpha\tau^{n}\left\Vert u\right\Vert
_{{n}}^{{n}}\right)  ^{\frac{1}{n-1}}\right)  dx\\
&  \geq\tau^{n}\int_{
\mathbb{R}
^{n}}\Phi\left(  \underset{k\rightarrow\infty}{\lim}\alpha_{k}u^{\frac{n}%
{n-1}}\right)  dx+o(1)\\
&  =\int_{
\mathbb{R}
^{n}}\left(  \Phi\left(  \underset{k\rightarrow\infty}{\lim}\alpha_{k}%
u^{\frac{n}{n-1}}\right)  +\left(  \tau^{n}-1\right)  \frac{\underset{k\rightarrow\infty}{\lim}\alpha_{k}^{n-1}u^{n}}{\left(
n-1\right)  !}\right)  dx\\
& \ \ +\left(  \tau^{n}-1\right)  \int_{
\mathbb{R}
^{n}}\left(  \Phi\left(  \underset{k\rightarrow\infty}{\lim}\alpha_{k}%
u^{\frac{n}{n-1}}\right)  -\frac{\underset{k\rightarrow\infty}{\lim}\alpha_{k}^{n-1}u^{n}}{\left(
n-1\right)  !}\right)dx  +o(1)\\
&  \geq\left(  \tau^{n}-1\right) \int_{
\mathbb{R}^{n}} \left(\Phi\left(  \underset{k\rightarrow\infty}{\lim}\alpha_{k}u^{\frac{n}%
{n-1}}\right)  -\frac{\underset{k\rightarrow\infty}{\lim}\alpha_{k}^{n-1}u^{n}}{\left(
n-1\right)  !}\right)dx  \\
& \ \ +\underset{k\rightarrow\infty}{\lim}\int_{
\mathbb{R}
^{n}}\Phi\left(  \alpha_{k}u_{k}^{\frac{n}{n-1}}\right)  dx+o(1)\\
&  =S+\left(  \tau^{n}-1\right)  \int_{
\mathbb{R}
^{n}}\left(  \Phi\left(  \underset{k\rightarrow\infty}{\lim}\alpha_{k}%
u^{\frac{n}{n-1}}\right)  -\frac{\underset{k\rightarrow\infty}{\lim}\alpha_{k}^{n-1}u^{n}}{\left(
n-1\right)  !}\right)  dx+o(1).
\end{align*}

Since $\Phi\left(  \underset{k\rightarrow\infty}{\lim}\alpha_{k}u^{\frac
{n}{n-1}}\right)  -\frac{\underset{k\rightarrow\infty}{\lim}\alpha_{k}^{n-1}u^{n}}{\left(
n-1\right)  !}>0$, we get $\tau=1$, thus
\[
\underset{k}{\lim}\int_{
\mathbb{R}
^{n}}\Phi\left(  \alpha_{k}u_{k}^{\frac{n}{n-1}}\right)  dx=\int_{
\mathbb{R}
^{n}}\Phi\left(  \lambda_{n}\left(  1+\alpha\left\Vert u\right\Vert _{{n}}
^{{n}}\right)  ^{\frac{1}{n-1}}u^{\frac{n}{n-1}}\right)  dx.
\]
So $u$ is an extremal function and the proof of Theorem \ref{theorem5.1} is fininshed.
\end{proof}

\subsection{\bigskip Proof in the case of $\sup_{k}c_{k}=+\infty$}
In this subsection, we will show that the existence of the extremal functions of
 Moser-Trudinger ineuqality involving $L^{n}$ norm in $
\mathbb{R}^{n}$ in the case of $\sup_{k}c_{k}=+\infty$. In order to prove the existence of the extremal functions, we need the
following result due to Zhou and Zhou \cite{ZZ1}, which
often plays a key role in the proof of existence result. This method has been widely used to prove the existence of the extremal functions of many kinds of Moser-Trudinger inequality (see \cite{liruf,Y2,lu-yang 1,zhu}).

\begin{lemma}
\bigskip\label{Z-Z} Assume that $\left\{  u_{k}\right\}$ is a normalized concentrating sequence in  $W^{1.n}_{0}(\mathcal{W}_{1})$
with a blow up point at the orgin, i.e.  $\int_{\mathcal{W}_{1}}F^{n}(\nabla u_{k})dx=1$, $u_{k}\rightharpoonup0$ weakly in $W_{0}^{1,n}(\mathcal{W}_{1})$
and $\lim\limits_{k\rightarrow +\infty}\int_{\mathcal{W}_{1}\backslash \mathcal{W}_{r}}F^{n}(\nabla u_{k})dx=0$ for any $0<r<1$, then
\[
\underset{k\rightarrow \infty}{\lim\sup}\int_{\mathcal{W}_{1}}e^{\lambda_{n}\left\vert
u_{k}\right\vert ^{\frac{n}{n-1}}dx}\leq \kappa_{n}\left( 1+\exp\left\{1+\frac{1}{2}
+\ldots+\frac{1}{n-1}\right\}\right).
\]
\end{lemma}

\begin{lemma}\label{lem5.2}
\label{attain lemma} If $S$ can not be attained, then
\[
S\leq \kappa_{n}\exp\left\{  \lambda_{n}A+1+\frac{1}%
{2}+\ldots+\frac{1}{n-1}\right\},
\]
where $A$ is the constant in (\ref{13}).
\end{lemma}

\begin{proof}
\ By Lemma \ref{tend to G}, it follows that
\begin{align*}
&  \bigskip\lim_{k\rightarrow\infty}\int_{\mathbb{R}^{n}\backslash \mathcal{W}_{\delta}}\left(  F^{n}(\nabla u_{k})
+\left\vert u_{k}\right\vert ^{n}\right)  dx\\
=&c_{k}^{\frac{-n}{n-1}}\left(  \alpha\int_{
\mathbb{R}
^{n}\backslash \mathcal{W}_{\delta}}U_{k}^{{n}}dx+G_{\alpha}(\delta)\left(1+(\alpha-1)\int_{\mathcal{W}_\delta}G_{\alpha}^{n-1}dx\right)\right)  \\
=&c_{k}^{\frac{-n}{n-1}}\left(  \alpha\underset{k\rightarrow\infty}{\lim
}\left\Vert U_{k}\right\Vert _{{n}}^{{n}}+G_{\alpha}(\delta)\left(1+(\alpha-1)\int_{\mathcal{W}_\delta}G_{\alpha}^{n-1}dx\right)\right)  .
\end{align*}

Set $\tilde{u}_{k}\left(  x\right)  =\left(  u_{k}\left(  x\right)
-u_{k}\left(  \delta\right)  \right)  ^{+}$, then
\begin{align}
\bigskip\bigskip\int_{\mathcal{W}_{\delta}}F^{n}( \nabla\tilde{u}_{k})
dx &  \leq\int_{\mathcal{W}_{\delta}}F^{n}(\nabla u_{k})dx=\tau_{k}:=1-\int_{\mathbb{R}
^{n}\backslash \mathcal{W}_{\delta}}\left( F^{n}( \nabla u_{k})+\left\vert u_{k}\right\vert ^{n}\right)  dx-\int_{\mathcal{W}_{\delta}}\left\vert
u_{k}\right\vert ^{n}dx\nonumber\\
&  =1-c_{k}^{\frac{-n}{n-1}}\left(  \alpha\underset{k\rightarrow\infty}{\lim
}\left\Vert U_{k}\right\Vert _{{n}}^{{n}}-\frac{n}{\lambda_{n}}\log
\delta+A+o_{k}\left(  1\right)  +O_{\delta}\left(  1\right)  \right)
.\label{18}
\end{align}

When $x\in \mathcal{W}_{Lr_{k}}$, by Lemma \ref{tend to G 1} and (\ref{18}), it follows that
\begin{align*}
&  \alpha_{k}u_{k}^{\frac{n}{n-1}}\leq\lambda_{n}\left(  1+\alpha\left\Vert
u_{k}\right\Vert _{{n}}^{{n}}\right)  ^{\frac{1}{n-1}}\left(  \tilde{u}
_{k}+u_{k}\left(  \delta\right)  \right)  ^{\frac{n}{n-1}}\\
\leq& \lambda_{n}\left\vert \tilde{u}_{k}\right\vert ^{\frac{n}{n-1}}
+\frac{n\lambda_{n}}{n-1}\left\vert \tilde{u}_{k}\right\vert ^{\frac{1}{n-1}
}\left\vert u_{k}\left(  \delta\right)  \right\vert +\frac{\lambda_{n}\alpha
}{n-1}\Vert c_{k}^{\frac{1}{n-1}}u_{k}\Vert _{{n}}^{{n}}
+o_{k}\left(  1\right)  \\
\leq&\lambda_{n}\left\vert \tilde{u}_{k}\right\vert ^{\frac{n}{n-1}}
+\frac{n\lambda_{n}}{n-1}\left\vert c_{k}\right\vert ^{\frac{1}{n-1}}\left\vert
u_{k}\left(  \delta\right)  \right\vert +\frac{\lambda_{n}\alpha}{n-1}
\underset{k\rightarrow\infty}{\lim}\left\Vert U_{k}\right\Vert _{{n}}^{{n}
}+o_{k}\left(  1\right)  \\
\leq&\lambda_{n}\left\vert \tilde{u}_{k}\right\vert ^{\frac{n}{n-1}}
+\frac{n\lambda_{n}}{n-1}\left\vert G_{\alpha}\left(  \delta\right)
\right\vert +\frac{\lambda_{n}\alpha}{n-1}\underset{k\rightarrow\infty}{\lim
}\left\Vert U_{k}\right\Vert _{{n}}^{{n}}+o_{k}\left(  1\right)  \\
=&\lambda_{n}\left\vert \tilde{u}_{k}\right\vert ^{\frac{n}{n-1}}-\frac{n^{2}
}{n-1}\log\delta+\frac{n\lambda_{n}}{n-1}A+\frac{\lambda_{n}\alpha}
{n-1}\underset{k\rightarrow\infty}{\lim}\left\Vert U_{k}\right\Vert _{{n}
}^{{n}}+o_{k}\left(  1\right)  +o_{\delta}\left(  1\right)  \\
\leq&\frac{\lambda_{n}\left\vert \tilde{u}_{k}\right\vert ^{\frac{n}{n-1}}
}{\tau_{k}^{\frac{1}{n-1}}}+\lambda_{n}A-\log\delta^{n}+o_{k}\left(  1\right)
+o_{\delta}\left(  1\right).
\end{align*}
Integrating the above estimates on $\mathcal{W}_{Lr_{k}}$, we obtain
\begin{align*}
\int_{\mathcal{W}_{Lr_{k}}}\left(  \exp\left\{  \alpha_{k}u_{k}^{\frac{n}{n-1}}\right\}
-1\right)  dx &  \leq\delta^{-n}\exp\left\{  \lambda_{n}A+o_{k}\left(
1\right)  \right\} \\
&  \cdot\int_{\mathcal{W}_{Lr_{k}}}\left(  \exp\left\{  \alpha_{k}u_{k}^{\frac{n}{n-1}
}/\tau_{k}^{\frac{1}{n-1}}\right\}  -1\right)  dx+o_{k}\left(  1\right).
\end{align*}
By Lemma \ref{Z-Z}, we get
\[
\int_{\mathcal{W}_{Lr_{k}}}\left(  \exp\left\{  \alpha_{k}u_{k}^{\frac{n}{n-1}}\right\}
-1\right)  dx\leq\kappa_{n}\exp\left\{  \lambda_{n}A+1+\frac{1}
{2}+\ldots+\frac{1}{n-1}\right\}.
\]
By Lemma \ref{lem4.6}, we obtain
\begin{align}
\underset{k\rightarrow\infty}{\lim}\int_{
\mathbb{R}
^{n}}\Phi\left(  \alpha_{k}u_{k}^{\frac{n}{n-1}}\right)  dx &  \leq
\underset{L\rightarrow\infty}{\lim}\underset{k\rightarrow\infty}{\lim}%
\int_{\mathcal{W}_{Lr_{k}}}\left(  \exp\left(  \alpha_{k}u_{k}^{\frac{n}{n-1}}\right)
-1\right)  dx\nonumber\\
\leq& \kappa_{n}\exp\left\{  \lambda_{n}A+1+\frac{1}{2}
+\ldots+\frac{1}{n-1}\right\}  .
\end{align}
Thus the conclusion of Lemma \ref{lem5.2} holds.
\end{proof}

In the following, we will construct a function sequence $\left\{
u_{\varepsilon}\right\}  \subset W^{1,n}\left(
\mathbb{R}
^{n}\right)  $ with $\left\Vert u_{\varepsilon}\right\Vert _{F}=1$ such that%

\[
\int_{
\mathbb{R}
^{n}}\Phi\left(  \lambda_{n}u_{\varepsilon}^{\frac{n}{n-1}}\right)
dx>\kappa_{n}\exp\left\{
\lambda_{n}A+1+\frac{1}{2}+\ldots+\frac{1}{n-1}\right\} .
\]
\begin{proof}
[Proof of Theorem \ref{attain} in the case of $\sup_{k}c_{k}=+\infty$] Let
\[
u_{\varepsilon}=\left\{
\begin{array}
[c]{c}%
\frac{C-C^{\frac{-1}{n-1}}\left(  \frac{n-1}{\lambda_{n}}\log\left(
1+c_{n}(\frac{F^{0}(x)}{\varepsilon})^{\frac{n}{n-1}}\right)
-B_{\varepsilon}\right)  }{\left(  1+\alpha C^{\frac{-n}{n-1}}\left\Vert
G_{\alpha}\right\Vert _{{n}}^{{n}}\right)  ^{\frac{1}{n}}},\text{
\ \ \ \ \ \ \ \ \ \ \ \ \ \ \ \ \ \ } F^{0}(x)\leq
R\varepsilon,\\
\frac{G_{\alpha}\left( F^{0}(x)\right)  }{\left(  C^{\frac
{n}{n-1}}+\alpha\left\Vert G_{\alpha}\right\Vert _{{n}}^{{n}}\right)
^{\frac{1}{n}}},\text{
\ \ \ \ \ \ \ \ \ \ \ \ \ \ \ \ \ \ \ \ \ \ \ \ \ \ \ \ \ } F^{0}(x)>R\varepsilon
,
\end{array}
\right.
\]
where $c_{n}=\kappa_{n}^{\frac{1}{n-1}}$,
$B_{\varepsilon}$, $R$ and $C$ depending on $\varepsilon$ will also be
determined later, such that

\medskip

\bigskip (i) \ $R\varepsilon\rightarrow0$, $R\rightarrow\infty$ and
$C\rightarrow\infty$, as $\varepsilon\rightarrow0,$

\medskip

(ii) $\ \frac{C-\frac{n-1}{\lambda_{n}}C^{\frac{-1}{n-1}}\log\left(
1+c_{n}\left\vert R\right\vert ^{\frac{n}{n-1}}\right)  +C^{\frac{-1}{n-1}}B_{\varepsilon}
}{\left(  1+\alpha C^{\frac{-n}{n-1}}\left\Vert G_{\alpha}\right\Vert _{{n}
}^{{n}}\right)  ^{\frac{1}{n}}}=\frac{G_{\alpha}\left(  R\varepsilon\right)
}{\left(  C^{\frac{n}{n-1}}+\alpha\left\Vert G_{\alpha}\right\Vert _{{n}}
^{{n}}\right)  ^{\frac{1}{n}}}.$

We can obtain the information of $B_{\varepsilon}$, $C$ and $R$ by
normalizating $u_{\varepsilon}$. By Lemma \ref{tend to G}, it can check that
\begin{align*}
&  \int_{\mathbb{R}^{n}\backslash \mathcal{W}_{R\varepsilon}}\left( F^{n}(\nabla u_{\varepsilon
})+\left\vert u_{\varepsilon}\right\vert ^{n}\right)  dx\\
=& \frac{1}{C^{\frac{n}{n-1}}+\alpha\left\Vert G_{\alpha}\right\Vert _{{n}
}^{{n}}}\int_{\mathbb{R}
^{n}\backslash \mathcal{W}_{R\varepsilon}}\left( F^{n}(\nabla G_{\alpha
})+\left\vert G_{\alpha}\right\vert ^{n}\right)  dx\\
=&\frac{1}{C^{\frac{n}{n-1}}+\alpha\left\Vert G_{\alpha}\right\Vert _{{n}
}^{{n}}}\left(  -G_{\alpha}\left(  R\varepsilon\right)  \int_{\partial
\mathcal{W}_{R\varepsilon}}\left( F ^{n-2}( \nabla G_{\alpha})
\frac{\partial G_{\alpha}}{\partial n}\right)  dx+\alpha\int_{
\mathbb{R}
^{n}\backslash \mathcal{W}_{R\varepsilon}}\left\vert G_{\alpha}\right\vert ^{n}dx\right)\\
=&\frac{n\kappa_{n}G_{\alpha}\left(  R\varepsilon\right) \left\vert
G^{\prime}\left(  R\varepsilon\right)  \right\vert ^{n-1}\left(
R\varepsilon\right)  ^{n-1}+\alpha\int_{
\mathbb{R}
^{n}\backslash \mathcal{W}_{R\varepsilon}}\left\vert G_{\alpha}\right\vert ^{n}%
dx}{C^{\frac{n}{n-1}}+\alpha\left\Vert G_{\alpha}\right\Vert _{{n}}^{{n}}},
\end{align*}
and
\begin{align*}
\int_{\mathcal{W}_{R\varepsilon}}F^{n}( \nabla u_{\varepsilon})dx
 &  =\frac{n-1}{\lambda_{n}\left(  C^{\frac{n}{n-1}}
+\alpha\left\Vert G_{\alpha}\right\Vert _{{n}}^{{n}}\right)  }\int_{0}
^{c_{n}R^{\frac{n}{n-1}}}\frac{u^{n-1}}{\left(  1+u\right)  ^{n}}du\\
&  =\frac{n-1}{\lambda_{n}\left(  C^{\frac{n}{n-1}}+\alpha\left\Vert G_{\alpha
}\right\Vert _{{n}}^{{n}}\right)  }\int_{0}^{c_{n}R^{\frac{n}{n-1}}}
\frac{\left(  \left(  1+u\right)  -1\right)  ^{n-1}}{\left(  1+u\right)  ^{n}
}du\\
&  =\frac{n-1}{\lambda_{n}\left(  C^{\frac{n}{n-1}}+\alpha\left\Vert G_{\alpha
}\right\Vert _{{n}}^{{n}}\right)  }\left(  \underset{k=0}{\overset{n-2}{\sum}
}\frac{C_{n-1}^{k}\left(  -1\right)  ^{n-1-k}}{n-k-1}+\right. \\
&  +\left.  \log\left(  1+c_{n}R^{\frac{n}{n-1}}\right)  +O\left(
R^{\frac{-n}{n-1}}\right)  \right).
\end{align*}

Using the fact that%

\[
E:=\underset{k=0}{\overset{n-2}{\sum}}\frac{C_{n-1}^{k}\left(  -1\right)
^{n-1-k}}{n-k-1}=-\left(  1+\frac{1}{2}+\frac{1}{3}+\cdots+\frac{1}%
{n-1}\right)  ,
\]
we obtain
\[
\int_{\mathcal{W}_{R\varepsilon}}F^{n}(\nabla u_{\varepsilon})dx=\frac{n-1}{\lambda_{n}\left(  C^{\frac{n}{n-1}}+\alpha
\left\Vert G_{\alpha}\right\Vert _{{n}}^{{n}}\right)  }\left(  E+\log\left(
1+c_{n}R^{\frac{n}{n-1}}\right)  +O\left(  R^{\frac{-n}{n-1}}\right)  \right).
\]

It is easy to check that
\[
\int_{\mathcal{W}_{R\varepsilon}} \left\vert u_{\varepsilon}\right\vert
^{n} dx=O(C^{n}(R\varepsilon)^{n}),
\]
then
\begin{align*}
&  \int_{\mathbb{R}
^{n}}\left( F^{n}( \nabla u_{\varepsilon})+\left\vert
u_{\varepsilon}\right\vert ^{n}\right)  dx=\frac{1}{\lambda_{n}\left(
C^{\frac{n}{n-1}}+\alpha\left\Vert G_{\alpha}\right\Vert _{{n}}^{{n}}\right)
}\left(  \left(  n-1\right)  E+\left(  n-1\right)  \log\left(  1+c_{n}
R^{\frac{n}{n-1}}\right) \right. \\
& -\log\left(  R\varepsilon\right)  ^{n}+\lambda_{n}A+\alpha\lambda
_{n}\left\Vert G_{\alpha}\right\Vert _{{n}}^{{n}}+O\left( \phi\right),
\end{align*}
where
\[
\phi= C^{n} \left(R\varepsilon\right)^{n}+\left(  R\varepsilon
\right)  ^{n}\log^{n}(R\varepsilon)+R^{\frac{-n}{n-1}}+C^{\frac{-2n}{n-1}%
}+C^{\frac{n^{2}}{n-1}}(R\varepsilon)^{n}  .
\]

Because $\int_{
\mathbb{R}
^{n}}\left(F^{n}(\nabla u_{\varepsilon})+\left\vert
u_{\varepsilon}\right\vert ^{n}\right)  dx=1$, it follows that

\begin{align*}
&  \lambda_{n}\left(  C^{\frac{n}{n-1}}+\alpha\left\Vert G_{\alpha}\right\Vert
_{{n}}^{{n}}\right)  =\left(  n-1\right)  E+\left(  n-1\right)  \log\left(
1+c_{n}R^{\frac{n}{n-1}}\right) \\
&  -\log\left(  R\varepsilon\right)  ^{n}+\lambda_{n}A+\alpha\lambda
_{n}\left\Vert G_{\alpha}\right\Vert _{{n}}^{{n}}+O\left( \phi\right),
\end{align*}
namely,
\begin{equation}
\lambda_{n}C^{\frac{n}{n-1}}=\left(  n-1\right)  E+\log\kappa_{n}-\log\varepsilon^{n}+\lambda_{n}A+O\left( \phi\right). \label{20}%
\end{equation}
By (ii) we obtain

\[
C-C^{\frac{-1}{n-1}}\left(  \frac{n-1}{\lambda_{n}}\log\left(  1+c_{n}
\left\vert R\right\vert ^{\frac{n}{n-1}}\right)  -B_{\varepsilon}\right)
=\frac{-\frac{n}{\lambda_{n}}\log(R\varepsilon)+A+O\left( \phi\right)}{C^{\frac{1}{n-1}}}.
\]
Thus
\begin{equation}
C^{\frac{n}{n-1}}=-\frac{n}{\lambda_{n}}\log\varepsilon+\log\kappa_{n}-B_{\varepsilon}+A+O\left( \phi\right). \label{21}%
\end{equation}
Combining (\ref{20}) and \bigskip(\ref{21}), it is easy to see that
\begin{equation}
B_{\varepsilon}=-\frac{n-1}{\lambda_{n}}E+O\left( \phi\right).\label{22}
\end{equation}

 Letting $R=-\log\varepsilon\,$, which satisfies $R\varepsilon
\rightarrow0$ as $\varepsilon\rightarrow0$, then
\begin{equation}
\left\Vert u_{\varepsilon}\right\Vert _{{n}}^{{n}}=\frac{\left\Vert G_{\alpha
}\right\Vert _{{n}}^{{n}}+O\left(  C^{\frac{n^{2}}{n-1}}R^{n}\varepsilon
^{n}\right)  +O\left((R\varepsilon)^{n}\left(  -\log\left(  R\varepsilon
\right)  ^{n}\right)  \right)  }{C^{\frac{n}{n-1}}+\alpha\left\Vert G_{\alpha
}\right\Vert _{{n}}^{{n}}}.\text{ } \label{add5}%
\end{equation}
It is easy to check that when $\left\vert t\right\vert <1$,

$$\left(  1-t\right)  ^{\frac{n}{n-1}}\geq1-\frac{n}{n-1}t,\ \
\left(  1+t\right)  ^{-\frac{1}{n-1}}\geq1-\frac{t}{n-1}.$$
\ By using the above inequalities and \bigskip(\ref{add5}), we deduce that for any $x\in
\mathcal{W}_{R\varepsilon}$,
\begin{align*}
&  \lambda_{n}\left\vert u_{\varepsilon}\right\vert ^{\frac{n}{n-1}}\left(
1+\alpha\left\Vert u_{\varepsilon}\right\Vert _{{n}}^{{n}}\right)  ^{\frac
{1}{n-1}}\\
=&\lambda_{n}C^{\frac{n}{n-1}}\frac{\left(  1-C^{\frac{-n}{n-1}}\left(
\frac{n-1}{\lambda_{n}}\log\left(  1+c_{n} (\frac{F^{0}(x)}{\varepsilon
})^{\frac{n}{n-1}}\right)  -B_{\varepsilon}\right)  \right)
^{\frac{n}{n-1}}}{\left(  1+\alpha C^{\frac{-n}{n-1}}\left\Vert G_{\alpha
}\right\Vert _{{n}}^{{n}}\right)  ^{\frac{1}{n-1}}}\left(  1+\alpha\left\Vert
u_{\varepsilon}\right\Vert _{{n}}^{{n}}\right)  ^{\frac{1}{n-1}}\\
\geq& \lambda_{n}C^{\frac{n}{n-1}}\left(  1-\frac{n}{n-1}C^{\frac{-n}{n-1}
}\left(  \frac{n-1}{\lambda_{n}}\log\left(  1+c_{n} (\frac{F^{0}(x)}{\varepsilon
}) ^{\frac{n}{n-1}}\right)  -B_{\varepsilon}\right)
\right)  \\
\cdot& \left(  1-\alpha C^{\frac{-n}{n-1}}\left\Vert G_{\alpha}\right\Vert
_{{n}}^{{n}}\right)  ^{\frac{1}{n-1}}\left(  1+\alpha\left\Vert u_{\varepsilon
}\right\Vert _{{n}}^{{n}}\right)  ^{\frac{1}{n-1}}\\
\geq&\lambda_{n}C^{\frac{n}{n-1}}\left(  1-\frac{n}{n-1}C^{\frac{-n}{n-1}
}\left(  \frac{n-1}{\lambda_{n}}\log\left(  1+c_{n} (\frac{F^{0}(x)}{\varepsilon
})^{\frac{n}{n-1}}\right)  -B_{\varepsilon}\right)
\right)  \\
\cdot&\left(  1-\alpha C^{\frac{-n}{n-1}}\left\Vert G_{\alpha}\right\Vert
_{{n}}^{{n}}\right)  ^{\frac{1}{n-1}}\left(  1+\alpha\frac{\left\Vert
G_{\alpha}\right\Vert _{{n}}^{{n}}+O\left(  C^{\frac{n^{2}}{n-1}}
R^{n}\varepsilon^{n}\right)  +O\left(  R^{n}\varepsilon^{n}\left(
-\log\left(  R\varepsilon\right)  ^{n}\right)  \right)  }{C^{\frac{n}{n-1}}
}\right)  ^{\frac{1}{n-1}}\\
\geq&\lambda_{n}C^{\frac{n}{n-1}}\left(  1-\frac{n}{n-1}C^{\frac{-n}{n-1}
}\left(  \frac{n-1}{\lambda_{n}}\log\left(  1+c_{n} (\frac{F^{0}(x)}{\varepsilon
})^{\frac{n}{n-1}}\right)  -B_{\varepsilon}\right)
\right) \\
\cdot&\left(  1-\alpha^{2}C^{\frac{-2n}{n-1}}\left\Vert G_{\alpha
}\right\Vert _{n}^{2n}+C^{\frac{-n}{n-1}}\left(  O\left(  C^{\frac{n^{2}}
{n-1}}R^{n}\varepsilon^{n}\right)  +O\left(  R^{n}\varepsilon^{n}\left(
-\log\left(  R\varepsilon\right)  ^{n}\right)  \right)  \right)  \right)
^{\frac{1}{n-1}}\\
\geq& \lambda_{n}C^{\frac{n}{n-1}}\left(  1-\frac{n}{n-1}C^{\frac{-n}{n-1}
}\left(  \frac{n-1}{\lambda_{n}}\log\left(  1+c_{n} (\frac{F^{0}(x)}{\varepsilon
})^{\frac{n}{n-1}}\right)  -B_{\varepsilon}\right)
\right)  \\
\cdot&\left(  1-\frac{1}{n-1}\alpha^{2}C^{\frac{-2n}{n-1}}\left\Vert
G_{\alpha}\right\Vert _{n}^{2n}+C^{\frac{-n}{n-1}}\left(  O\left(
C^{\frac{n^{2}}{n-1}}R^{n}\varepsilon^{n}\right)  +O\left(  R^{n}
\varepsilon^{n}\left(  -\log\left(  R\varepsilon\right)  ^{n}\right)  \right)
\right)  \right)  \\
\geq& \lambda_{n}C^{\frac{n}{n-1}}-n\log\left(  1+c_{n} (\frac{F^{0}(x)}{\varepsilon
})^{\frac{n}{n-1}}\right)  +\frac{n\lambda_{n}}
{n-1}B_{\varepsilon}-\frac{\lambda_{n}\alpha^{2}\left\Vert G_{\alpha
}\right\Vert _{n}^{2n}}{\left(  n-1\right)  C^{\frac{n}{n-1}}}+O\left( \phi\right).
\end{align*}
By (\ref{20}) and (\ref{22}), we obtain
\begin{align*}
&  \lambda_{n}\left\vert u_{\varepsilon}\right\vert ^{\frac{n}{n-1}}\left(
1+\alpha\left\Vert u_{\varepsilon}\right\Vert _{{n}}^{{n}}\right)  ^{\frac
{1}{n-1}}\\
\geq&  -E+\log\kappa_{n}-\log\varepsilon^{n}-n\log\left(
1+c_{n} (\frac{F^{0}(x)}{\varepsilon
}) ^{\frac{n}{n-1}}\right)  \\
&  -\frac{\lambda_{n}\alpha^{2}\left\Vert G_{\alpha}\right\Vert _{{n}}^{2{n}}
}{\left(  n-1\right)  C^{\frac{n}{n-1}}}+\lambda_{n}A+O\left( \phi\right).
\end{align*}
Therefore

\begin{align*}
&  \int_{\mathcal{W}_{R\varepsilon}}\Phi\left(  \lambda_{n}\left\vert u_{\varepsilon
}\right\vert ^{\frac{n}{n-1}}\left(  1+\alpha\left\Vert u_{\varepsilon
}\right\Vert _{{n}}^{{n}}\right)  ^{\frac{1}{n-1}}\right)  dx\\
\geq&\exp\left\{  -E+\lambda_{n}A+\log\kappa_{n}-\log
\varepsilon^{n}-\frac{\lambda_{n}\alpha^{2}\left\Vert G_{\alpha}\right\Vert
_{{n}}^{2{n}}}{\left(  n-1\right)  C^{\frac{n}{n-1}}}+O\left( \phi\right)\right\} \\
\cdot&\int_{\mathcal{W}_{R\varepsilon}}\exp\left\{  -n\log\left(  1+c_{n} (\frac{F^{0}(x)}{\varepsilon
}){\frac{n}{n-1}}\right)  \right\} \\
\geq& c_{n}^{n-1}\varepsilon^{-n}\exp\left\{  -E+\lambda_{n}A-\frac
{\lambda_{n}\alpha^{2}\left\Vert G_{\alpha}\right\Vert _{{n}}^{2{n}}}{\left(
n-1\right)  C^{\frac{n}{n-1}}}+O\left( \phi\right)\right\}  \int_{\mathcal{W}_{R\varepsilon}}\left(
1+c_{n} (\frac{F^{0}(x)}{\varepsilon
})^{\frac{n}{n-1}}\right)
^{-n}dx\\
\geq&(n-1)\kappa_{n}\exp\left\{  -E+\lambda
_{n}A-\frac{\lambda_{n}\alpha^{2}\left\Vert G_{\alpha}\right\Vert _{{n}}^{2{n}
}}{\left(  n-1\right)  C^{\frac{n}{n-1}}}+O\left( \phi\right)\right\}  \int_{0}
^{c_{n}R^{\frac{n}{n-1}}}\frac{u^{n-2}}{\left(  1+u\right)  ^{n}}du\\
\geq& (n-1)\kappa_{n}\exp\left\{  -E+\lambda
_{n}A-\frac{\lambda_{n}\alpha^{2}\left\Vert G_{\alpha}\right\Vert _{{n}}^{2{n}
}}{\left(  n-1\right)  C^{\frac{n}{n-1}}}+O\left( \phi\right)\right\}  \left(  \frac{1}
{n-1}+o\left(  R^{\frac{-n}{n-1}}\right)  \right) \\
\geq&\kappa_{n}\exp\left\{  -E+\lambda_{n}A\right\}  \left(
1-\frac{\lambda_{n}\alpha^{2}\left\Vert G_{\alpha}\right\Vert _{{n}}^{2{n}}
}{\left(  n-1\right)  C^{\frac{n}{n-1}}}+O\left( \phi\right)\right).
\end{align*}

Also
\begin{align*}
\int_{
\mathbb{R}
^{n}\backslash \mathcal{W}_{R\varepsilon}}\Phi\left(  \lambda_{n}u_{\varepsilon}
^{\frac{n}{n-1}}\right)  dx  &\geq\frac{\lambda_{n}^{n-1}}{\left(  n-1\right)
!C^{\frac{n}{n-1}}}\int_{
\mathbb{R}
^{n}\backslash \mathcal{W}_{R\varepsilon}}\left\vert G_{\alpha}\right\vert ^{n}dx\\
&=\frac{\lambda_{n}^{n-1}\left\Vert G_{\alpha}\right\Vert _{n}^{n}+O\left(
R^{n}\varepsilon^{n}\left(  \log\left(  R\varepsilon\right)  ^{n}\right)
\right)  }{\left(  n-1\right)  !C^{\frac{n}{n-1}}},
\end{align*}
thus
\begin{align*}
&  \int_{
\mathbb{R}
^{n}}\Phi\left(  \lambda_{n}\left\vert u_{\varepsilon}\right\vert ^{\frac
{n}{n-1}}\left(  1+\alpha\left\Vert u_{\varepsilon}\right\Vert _{n}
^{n}\right)  ^{\frac{1}{n-1}}\right)  dx\\
\geq&\kappa_{n}\exp\left\{  -E+\lambda_{n}A\right\}  \left(
1-\frac{\lambda_{n}\alpha^{2}\left\Vert G_{\alpha}\right\Vert _{{n}}^{2{n}}
}{\left(  n-1\right)  C^{\frac{n}{n-1}}}+O\left( \phi\right)\right)  +\frac{\lambda_{n}
^{n-1}\left\Vert G_{\alpha}\right\Vert _{n}^{n}}{\left(  n-1\right)
!C^{\frac{n}{n-1}}}.
\end{align*}

Since $R=\log\frac{1}{\varepsilon}$, by (\ref{21}) one can obtain $R\sim C^{\frac
{n}{n-1}}$, then it is easy to verify that $\phi=o\left(  C^{\frac{-n}{n-1}
}\right)  $. Thus when $\alpha$ small enough, we have

\[
\int_{
\mathbb{R}
^{n}}\Phi\left(  \lambda_{n}\left\vert u_{\varepsilon}\right\vert ^{\frac
{n}{n-1}}\left(  1+\alpha\left\Vert u_{\varepsilon}\right\Vert _{{n}}^{{n}%
}\right)  ^{\frac{1}{n-1}}\right)
dx>\kappa_{n}\exp\left\{ -E+\lambda_{n}A\right\}.
\]
Then the proof of Theorem \ref{attain} with $\sup_{k}c_{k}=+\infty$ has been completed.
\end{proof}

{\bf Acknowledgement.} The author would like to thank the supervisor Professor Jiayu Li for his continuous guidance and encouragement.
The research was partially supported by Natural Science Foundation of China (Nos.11526212, 11721101, 11971026),
Natural Science Foundation of Anhui Province (No.1608085QA12), Natural Science Foundation of Education Committee of Anhui Province (Nos.KJ2016A506, KJ2017A454)
and Excellent Young Talents Foundation of Anhui Province (No.GXYQ2017070).

\end{document}